\documentclass[nohyperref]{article}

\usepackage{microtype}
\usepackage{subfigure}
\usepackage{booktabs} 
\usepackage{graphicx}
\usepackage{caption}
\usepackage{natbib}
\usepackage{hyperref}

\usepackage[accepted]{icml2022}

\usepackage{amsmath}
\usepackage{amssymb}
\usepackage{mathtools}
\usepackage{amsthm}

\usepackage[capitalize,noabbrev]{cleveref}

\newcommand{\Ltwo}{{L^2(\Omega,\R)}}
\newcommand{\ltwo}{{L^2(\Omega)}}

\newcommand{\Hk}{{H^k(\Omega,\R)}}
\newcommand{\hk}{{H^k(\Omega)}}

\newcommand{\Hmk}{{H^{-k}(\Omega,\R)}}
\newcommand{\hmk}{{H^{-k}(\Omega)}}

\newcommand{\Ds}{\mathcal{D}'(\Omega)}

\newcommand{\K}{\mathbb{K}}

\newcommand{\muik}[1]{{|#1| \leq k}}

\newcommand{\J}{\mathbb{J}}
\newcommand{\D}{\mathbb{D}}
\newcommand{\I}{\mathbb{I}}

\newcommand{\R}{\mathbb{R}}

\newcommand{\N}{\mathbb{N}}
\newcommand{\Z}{\mathbb{Z}}
\newcommand{\T}{\mathbb{T}}

\newcommand{\W}{\mathbb{W}}

\newcommand{\Lc}{\mathcal{L}}

\newcommand{\Ic}{\mathcal{I}}

\newcommand{\lo}{\longrightarrow}
\newcommand{\li}{\left}
\newcommand{\re}{\right}

\newcommand{\im}{\mathrm{im}}

\newcommand{\dom}{\mathrm{dom}}

\theoremstyle{plain}
\newtheorem{theorem}{Theorem}
\newtheorem{proposition}[theorem]{Proposition}
\newtheorem{lemma}[theorem]{Lemma}
\newtheorem{corollary}[theorem]{Corollary}
\theoremstyle{definition}
\newtheorem{definition}[theorem]{Definition}

\theoremstyle{remark}
\newtheorem{remark}[theorem]{Remark}
\newtheorem{example}[theorem]{Example}
\newtheorem{experiment}{Experiment}[section]

\usepackage[textsize=tiny]{todonotes}

\icmltitlerunning{Learning Partial Differential Equations by Spectral Approximates of General Sobolev Spaces}

\begin{document}

\onecolumn
\icmltitle{Learning Partial Differential Equations by Spectral Approximates of General Sobolev Spaces}



\icmlsetsymbol{equal}{*}

\begin{icmlauthorlist}
\icmlauthor{Juan Esteban Suarez Cardona}{to}
\icmlauthor{Michael Hecht}{to}
\end{icmlauthorlist}

\icmlaffiliation{to}{CASUS - Center for Advanced System Understanding,  Helmholtz-Zentrum Dresden-Rossendorf e.V. (HZDR), G\"{o}rlitz, Germany}

\icmlcorrespondingauthor{Juan Esteban Suarez Cardona}{j.suarez-cardona@hzdr.de}
\icmlcorrespondingauthor{Michael Hecht}{m.hecht@hzdr.de}

\icmlkeywords{partial differential equations, spectral methods, machine learning}

\vskip 0.3in



\printAffiliationsAndNotice{}  

\begin{abstract}
  \parindent 24pt
  \parskip 1pt
  \vspace{3pt}
  We introduce a novel spectral, finite-dimensional approximation of general Sobolev spaces in terms of Chebyshev polynomials.  Based on this polynomial surrogate model (PSM), we realise a variational formulation, solving a vast class of linear and non-linear partial differential equations (PDEs). The PSMs are as flexible as the physics-informed neural nets (PINNs) and  provide an alternative for addressing inverse PDE problems, such as PDE-parameter inference. In contrast to PINNs, the PSMs result in a convex optimisation problem for a vast class of PDEs, including all linear ones, in which case the PSM-approximate is efficiently computable due to the exponential convergence rate of the underlying variational gradient descent.

  As a practical consequence prominent PDE problems were resolved by the PSMs \emph{without High Performance Computing} (HPC) on a local machine.
  This gain in efficiency is complemented by an increase of approximation power, outperforming PINN alternatives in both accuracy and runtime.

  Beyond the empirical evidence we give here, the  translation of classic PDE theory in terms of the Sobolev space approximates suggests the PSMs to be universally applicable to well-posed, regular forward and inverse PDE problems.
\end{abstract}

\section{Introduction}
Partial differential equations (PDEs) are omnipresent mathematical models governing the dynamics and (physical) laws of complex systems \cite{jost2002,brezis2011}.
However, analytic PDE solutions are rarely known for most of the systems being the centre of current research.
Therefore, there is a strong demand on efficient and accurate numerical solvers and simulations.\\
Main classic numerical solvers divide into: Finite Elements \cite{ern2004theory}; Finite Differences \cite{LeVeque2007FiniteDM}; Finite Volumes\cite{eymard2000finite}; Spectral Methods \cite{bernardi1997spectral,canuto2007spectral} and Particle Methods \cite{li2007meshfree}.\\
Machine learning methods such as:  Physics-Informed GAN \cite{pmlr-v70-arjovsky17a},
Deep Galerkin Method \cite{Sirignano2018DGMAD}, and Physics Informed Neural Networks (PINNs) \cite{RAISSI2019686}, gain big traction in the scientific computing community.
In contrast to classic solvers, PINNs provide a neural net (NN) surrogate model e.g.,  $\hat u : (-1,1)^m \lo \R$, $m \in \N$, parametrising the solution space of the PDEs and enabling to solve \emph{inverse problems} like inference of PDE parameters  or initial condition detection.  PINN-learning is given by minimising a
variational problem, which is typically formulated in $L^2$-loss terms
\begin{equation}\label{L2}
    \int_{\Omega} \big|\hat u(x) - u(x)\big|^2 d\Omega \approx \frac{1}{|P|}\sum_{p \in P}\big|\hat u(p) - u(p)\big|^2
\end{equation}
being approximated by the mean square error (MSE) in random (data) nodes $P$, \cite{Yang2020PhysicsInformedGA},\cite{Long2018PDENetLP}. The applications of PINNs range from fluid mechanics \cite{jin_nsfnets_2020} to biology \cite{lagergren_biologically-informed_2020} or medicine \cite{sahli_costabal_physics-informed_2020}, physics \cite{ellis2021accelerating} and beyond.

\subsection{Related work -- Physics Informed Neural Nets (PINNs)}
We identify the essential approaches addressing stability and accuracy of PINNs below.

\subsubsection{Variational PINNs (VPINNs)}\label{sec:VPINN}
VPINNs were introduced in  \cite{kharazmi2019variational,Kharazmi2020hpVPINNsVP}
resting on variational Sobolev losses for PINN-training. The approach exploits analytic integration and differentiation formulas of shallow neural networks with specified  activation functions. The method is extended by using quadrature rules and automatic differentiation for computing the losses and is complemented by a domain decomposition approach. The drawback of VPINNs, we identify and demonstrate here, is their highly consuming runtime performance, preventing the approach to be applicable for multi-dimensional PDE problems.

\subsubsection{Inverse Dirichlet loss balancing}\label{sec:ID}
The Inverse Dirichlet method \cite{maddu2021} was shown to increase the numerical
stability of PINNS by dynamically balancing the
occurring variational gradient amplitudes, which if unbalanced
cause numerical stiffness phenomena  \cite{wang2021understanding}. However, the PINN formulation rests on classic MSE losses, limiting the approach to consider only strong PDE problem formulations.

\subsubsection{Sobolev Cubatures PINNs (SC-PINN)}
In our prior work \cite{cardona2022replacing} we gave a PINN formulation, by replacing the MSE loss by \emph{Sobolev Cubatures}. In contrast to ID-PINNs approximating Sobolev losses enables the approach to consider PDE problems in  the weak and strong sense.
As a consequence,  the \emph{automatic differentiation} (A.D.) is replaced by \emph{polynomial differentiation} implicitly realised in the Sobolev cubatures. As we demonstrated this results in an increase of accuracy and runtime efficiency by several orders of magnitude compared to PINNs relying on A.D.

\subsection{Related Work - Classic spectral methods}
Spectral methods are well established techniques solving PDEs and ODEs. Hereby, one aims to approximate the PDE solution by an expansion $u = \sum_{\alpha \in A}c_\alpha \varphi_\alpha$, $A\subseteq \N^m$ with respect to a specific finite dimensional space $\Pi =\mathrm{span}\{\varphi_\alpha\}_{\alpha \in A}$ generated by a chosen basis, e.g., Fourier basis for periodic PDEs or Jacobi-Chebyshev polynomials for general, non-periodic problems. The coefficients of the expansion are  constrained by
the PDE and its corresponding boundary conditions. For example: Consider a (non-linear) differential operator $L$ and the equation
\begin{equation*}
    Lu = f \quad \text{in} \,\, \Omega,
\end{equation*}
with homogeneous Dirichlet boundary conditions. By sampling the function $\mathfrak{f} = f(p_\alpha)_{\alpha \in A}\in \R^{|A|}$, $A\subseteq \N^m$ in some node set $P = \{p_{\alpha}\}_{\alpha \in A}$ determination of the coefficients $C := (c_\alpha)_{\alpha \in A}\subseteq \mathbb{R}^{|A|}$ demands solving the truncated (non-linear) system:
\begin{equation*}
\mathbb{L}[C] - \mathfrak{f}\stackrel{!}{=}0\,,
\end{equation*}
where $\mathbb{L} = L_{|\Pi}$  denotes the truncated operator.
This system of equations is typically formulated as the solution of the weighted residual:
\begin{equation*}
    \langle\varphi_i, \mathbb{L}[C] - \mathfrak{f}\rangle \stackrel{!}{=} 0 \,, \quad \forall \alpha \in A.
\end{equation*}
Depending on the choice of the test functions $\varphi_i$ we obtain \emph{pseudo-spectral methods} or \emph{Galerkin spectral methods} \cite{Kang2008,canuto2007spectral,bernardi1997spectral}. If the operator $\mathbb{L}$ is linear, the problem is reduced to solving a linear system. In the non-linear case, least square methods with \emph{Newton-Raphson minimiser} are commonly used \cite{HESSARI2013318,kim2006}. Extending this formulation to  \emph{inverse problems} (inferring parameters) with \emph{general boundary conditions} and/or additional constraints without causing \emph{ill-conditioned problems} is a unresolved  challenge for classic spectral methods.
Our contribution relies on providing the demanded extensions, enabling  to addresses general forward and inverse PDE problems in a numerically stable, efficient and accurate fashion.

\subsection{Contribution}

 We present  a generalised \emph{soft-constrained spectral method} that results in a $\lambda$-convex variational optimisation problem for linear and a class of non-linear PDEs. We theoretically guarantee exponentially fast convergence of the resulting variational gradient descent. While established PINN alternatives result in non-convex variational problems, already for linear PDEs, the
spectral \emph{polynomial surrogate models} (PSMs) provide approximates of the PDE solutions outperforming PINNs in runtime and accuracy, as demonstrated in Section~\ref{sec:Num}.

Our approach rests on using \emph{Chebyshev Polynomial Surrogate Models (PSMs)}:
\begin{equation}\label{PSM}
    \hat{u}(x,\Theta) = \sum\limits_{\alpha\in A_{m,n}}\theta_\alpha T_\alpha(x)\,, \quad \Theta=(\theta_\alpha)_{\alpha\in A_{m,n}} \in \R^{|A_{m,n}|}, x \in \R^m\,,
\end{equation}
where $A_{m,n}$ denotes a multi-index set, see Section~\ref{sec:not}, and
$T_\alpha$ denotes the \emph{Cheybshev polynomial basis of first kind}  given by the relation:
\begin{equation}\label{eq:ChebP}
    T_\alpha(\cos(x)) = T_\alpha(\cos(x_1), \dots, \cos(x_m)) = \prod_{i=1}^m \cos(\alpha_i x_i) = \cos(\alpha x)
\end{equation}
for all $\alpha \in A_{m,n}$. The Chebyshev polynomials are widely used due to their excellent approximation properties extensively discussed in \cite{Trefethen2019}. In our recent work \cite{cardona2022replacing}, we already formulated (weak) PDE losses by generalising classic Gauss-Legendre cubature rules,  we termed \emph{Sobolev cubatures}. As aforementioned, for linear and a  class of non-linear PDEs the induced variational $\lambda$-convex gradient flows possess an exponential rate of convergence. The resulting PSMs deliver an increase of
accuracy up to $10$ orders of magnitude, by reducing the runtime costs up to $3$ orders of magnitude compared to PINN alternatives.
Moreover, we demonstrate the PSMs to be as flexible as PINNs for addressing  inverse PDE problems, such as  PDE-parameter inference.

In contrast to PINNs, the prominent PDE problems considered in Section~\ref{sec:Num} were solved by our PSM-method
\emph{without High Performance Computing} (HPC) on a local machine. We consequently expect the approach to deeply impact current methodology addressing  computational challenges arising across all scientific disciplines and believe that even currently non-reachable (high-dimensional, strongly varying) PDE problems can be successfully resolved due to our contribution.

\section{PDE theory}
In this section we introduce the mathematical concepts  on which our approach rest. This includes the formulation of Sobolev cubatures  \cite{cardona2022replacing}, approximating general Sobolev norms. To start with we fix the notation used throughout this article.

\subsection{Notation and basic concepts}\label{sec:not}

We denote with $\Omega=(-1,1)^m$ the open $m$-dimensional \emph{standard hypercube}, with $\bar \Omega = [-1,1]^m$ its closure, and with $\partial \Omega$ its boundary. $\|x\|_p = (\sum_{i=1}^m |x_i|^p)^{1/p}$, $x=(x_1,\ldots,x_m) \in \R^m$, $1 \leq p < \infty$, $\|x\|_{\infty} = \max_{1\leq i\leq m} |x_i|$ denotes the $l_p$-norm, and  $\left<x,y\right>$, $\|x\|$, $x,y \in \R^m$ the standard Euclidean inner product and norm on $\R^m$.

Moreover, $\Pi_{m,n} = \mathrm{span}\{x^\alpha\}_{\|\alpha\|_\infty \leq n}$ denotes the $\R$-\emph{vector space of all real polynomials} in $m$ variables spanned by all monomials
$x^{\alpha} = \prod_{i=1}^mx_i^{\alpha_i}$ of \emph{maximum degree} $n \in \N$, whereas $\Pi_{m,n}(\partial\Omega) = \{Q_{| \Omega} : Q \in \Pi_{m,n}\}$ denotes the space of restricted polynomials with support $\Omega$.

We consider the multi-index set
$A_{m,n}=\{\alpha \in \N^m : \|\alpha\|_\infty \leq n\}$  with $|A_{m,n}| = (n+1)^m$ and order $A_{m,n}$ with respect to the \emph{lexicographic order} $\preceq$ on $\N^m$ starting from last entry to the $1$st, e.g.,
$(5,3,1)\preceq (1,0,3) \preceq (1,1,3)$.
Let $\D \in \R^{|A_{m,n}|\times |A_{m,n}|}$ be a matrix we  slightly abuse notation by writing
\begin{equation}
 \D = (d_{\alpha,\beta})_{\alpha,\beta \in A_{m,n}}\,,
\end{equation}
where $d_{\alpha,\beta} \in \R$ is the $\alpha$-th, $\beta$-th entry of $\D$.

\subsection{Sobolev space theory}
We recommend \cite{Adams2003,parima2008,brezis2011} for an excellent overview on functional analysis and Sobolev space theory including the concepts we shortly  summarise: We denote with $C^k(\Omega,\R)$, $k \in \N \cup\{\infty\}$ the
\emph{Banach spaces}
of all $k$-times continuously differentiable functions with norm $\|f\|_{C^k(\Omega)} = \sum_{i=0}^k \sup_{x \in \Omega,\|\alpha\|_1 =i}|D^\alpha f(x)|$.
The Sobolev spaces
\begin{equation*}
    \begin{split}
      H^k(\Omega,\R) &= \li\{ f \in L^2(\Omega,\R) : D^\alpha f \in L^2(\Omega,\R)\re\}\,, \\
    \end{split}
\end{equation*}
$\|\alpha\|_1 = \sum_{i=1}^m \alpha_i \leq k$, $k \in \N$ are
given by all $L^2$-integrable functions $f : \Omega \lo \R$ with existing  $L^2$-integrable weak derivatives $D^\alpha f= \partial^{\alpha_1}_{x_1}\ldots\partial^{\alpha_m}_{x_m}f$ up to order $k$. In fact,
$H^k(\Omega,\R)$ is a Hilbert space with inner product
$$
    \li<f,g\re>_{H^k(\Omega)} =  \sum_{0 \leq \|\alpha\|_1 \leq k}\li<D^\alpha f,D^\alpha g\re>_{L^2(\Omega)}
$$
and norm $\|f\|_{H^k(\Omega)}^2 = \li<f,f\re>_{H^k(\Omega)}$. Thus, the embeddings
$j : H^{k}(\Omega,\R) \hookrightarrow  H^{k'}(\Omega)$ are well defined and continuous for all $k' \leq k$ due to
$\|\cdot\|_{H^{k'}(\Omega} \leq \|\cdot\|_\Hk$, whereas $H^0(\Omega,\R) = L^2(\Omega,\R)$, with $\li<f,g\re>_{L^2(\Omega)} = \int \limits_{\Omega}f\cdot g \, d\Omega$.

For $k\geq 1$ the trace operator
\begin{equation}\label{eq:tr}
    \mathrm{tr}: H^k(\Omega,\R) \lo L^2(\partial \Omega,\R)
\end{equation}
is defined as usual as the $H^k$-extension of the classic continuous trace $\mathrm{tr}(u) = u_{| \partial \Omega}$ with domain $\mathrm{dom}(\mathrm{tr})= C^0(\bar \Omega,\R)$.
The Sobolev spaces with zero trace are denoted as usual with $H^k_0(\Omega,\R) = \{u \in H^k(\Omega,\R) : \mathrm{tr}(u) =0\}$, $k \geq 1$ and can be alternatively defined as completion of the space of smooth functions that vanish on the boundary $\partial \Omega$ of $\Omega$, i.e.,
 $$H^k_0(\Omega,\R) = \overline{ C^\infty_0(\Omega,\R)}^{\|\cdot\|_{H^k(\Omega)}}\,, \quad
C^\infty_0(\Omega,\R) = \{f \in C^\infty(\Omega,\R): f_{|\partial \Omega} =0\}\,.$$

We further consider the space of all distributions $\Ds=\{F : C^\infty_0(\bar \Omega) \lo \R\}$ also known as \emph{generalised functions} (being the dual space of all test functions $C^\infty_0(\bar \Omega)= \{f \in C^\infty(\Omega) : f_{|\partial \Omega} = 0 \}$ with respect to the \emph{canonical LF topology}).
We associate the
\emph{negative order Sobolev space} as the completion of $\Ds$
with respect to the following norm
\begin{equation}\label{eq:Hkm_norm}
    H^{-k}(\Omega,\R) := \overline{\Ds}^{\|\cdot\|_{H^{-k}(\Omega)}}\,, \quad
    \|F \|_\Hmk = \sup_{u \in \Hk} \frac{|Fu|}{\| u \|_\Hk}\,,
\end{equation}
yielding a separable, reflexive Hilbert space \cite{lax1955}.

The weak PDE formulations and their underlying Hilbert space choice we will propose later on require the notion of \emph{adjoint (differential) operators}. We recall the  definition.

\begin{definition}[Adjoint operators]
    \label{gen_adj}
     Let $(K, \|\cdot\|_K), (H, \|\cdot\|_H)$ be Hilbert spaces and $T: \dom(T) \subseteq K \lo H$, $T^*: \dom(T^*)\subseteq H \lo K$ be linear operators with dense domains. Then $T^*$ is called an adjoint operator of
     $T$ if and only if
    \begin{equation*}
        \langle Tx, y \rangle_H = \langle x, T^*y \rangle_K
    \end{equation*}
    for all $x \in \dom(T)$ and $y \in \dom(T^*)$.
\end{definition}
\begin{example}
    Consider  $\partial_{x_i} : \Ltwo \lo \Ltwo$ as the  differential operator in the weak sense.
    Then its domain is given by  $\dom(\partial_{x_i}) = H^1(\Omega, \R) \subseteq L^2(\Omega,\R)$, which is a dense subset. Following  Definition~\ref{gen_adj}, and applying integration by parts, an adjoint operator  $\partial_{x_i}^*: \Ltwo \lo \Ltwo$, with domain $\dom(\partial_{x_i}^*) = H_0^1(\Omega, \R)$ is given by
    $\partial_{x_i}^* = - \partial_{x_i}$.
\end{example}
We link the spaces $H^{-k}(\Omega,\R)$ and $H^k(\Omega,\R)$ due to the following fact.
\begin{proposition}\label{prop:Hmk} Let  $j: \Hk \hookrightarrow \Ltwo$, $k \in \N$ be the embedding with adjoint operator $j^* : \Ltwo \lo \Hk$.
    Let $f,g \in \Ltwo$ and the distributions
    $F = \langle f, \cdot \rangle_\Ltwo, G = \langle g, \cdot \rangle_\Ltwo  \in \Hmk$, with $f \in \Ltwo$. Then
    \begin{equation*}
        \|F\|_\Hmk = \|j^* f\|_\hk\,, \quad \langle F, G\rangle_{\hmk} = \langle j^*f, j^* g\rangle_\hk\,.
    \end{equation*}
\end{proposition}

\begin{proof}
The proof is derived directly from the definition of the $\Hmk$-norm
in Eq.~\eqref{eq:Hkm_norm}:
\begin{align*}
    \|j^* f\|_\hk &= \frac{\|j^* f\|_\hk^2}{\|j^* f\|_\hk} = \frac{|\langle j f, j^* f \rangle_\ltwo|}{\|j^* f\|_\hk}  = \frac{|\langle f, j^* f \rangle_\ltwo|}{\|j^* f\|_\hk}\\
                  & \leq \sup_{u \in \Hk} \frac{|\langle f, u \rangle_\ltwo|}{\| u\|_\hk}=  \|F\|_\hmk\,.
\end{align*}
Vice versa, applying the Cauchy-Schwarz inequality yields
    \begin{align*}
        \|F\|_\Hmk &= \sup_{u \in \Hk} \frac{|\langle f, j u \rangle_\ltwo|}{\| u \|_\hk} = \sup_{u \in \Hk} \frac{|\langle j^* f, u \rangle_\hk|}{\| u \|_\hk} \\
        &\leq \sup_{u \in \Hk} \frac{\|j^* f\|_\hk \|u\|_\hk}{\| u \|_\hk} = \|j^*f\|_\hk\,,
    \end{align*}
implying the claimed equality. The statement for the inner product follows analogously.
\end{proof}

A main ingredient of all further considerations are the truncated $L^2$- or $H^k$-inner products that rest on adaptions of classic Gauss-Legendre cubatures, which we provide next.
\subsection{Orthogonal polynomials and Gauss-Legendre cubatures}

Here, we recapture the underlying concept of orthogonal polynomials:
Let $m,n\in \N$ and $P_{m,n} = \oplus_{i=1}^m \mathrm{Leg_n} \subseteq \Omega$ be the we the $m$-dimensional Legendre grids, where $\mathrm{Leg_n}=\{p_0,\ldots,p_n\}$ are the $n+1$ \emph{Legendre nodes} given by the roots of the \emph{Legendre polynomials} of degree $n+2$
We  denote $p_{\alpha} = (p_{\alpha_1}, \ldots, p_{\alpha_m}) \in P_{m,n}$, $\alpha \in A_{m,n}$. It is a classic fact \cite{stroud,stroud2,trefethen2017,Trefethen2019}, that the Lagrange polynomials $L_{\alpha}\in \Pi_{m,n}$, $\alpha \in A_{m,n}$  given by
\begin{equation}\label{eq:Lag}
    L_{\alpha} = \prod_{i=1}^ml_{\alpha_i,i}\,, \quad l_{j,i} = \prod_{j \not = i, j=0}^m \frac{x_i -p_j}{p_i-p_j}\,,
\end{equation}
satisfy $L_{\alpha}(p_\beta)=\delta_{\alpha,\beta}$, $\forall \, \alpha,\beta \in A_{m,n}$ and
form an orthogonal $L^2$-basis of $\Pi_{m,n}$, i.e.,
$$\li<L_{\alpha},L_{\beta}\re>_{L^2(\Omega)}=\int \limits_{\Omega} L_{\alpha}(x)L_{\beta}(x)d\Omega = w_{\alpha} \delta_{\alpha,\beta}\,,$$
$\forall \, \alpha,\beta \in A_{m,n}$, where $\delta_{\cdot,\cdot}$ denotes the \emph{Kronecker delta} and
\begin{equation}\label{eq:GLW}
    w_{\alpha} = \|L_{\alpha}\|^2_{L^2(\Omega)}
\end{equation} the efficiently computable \emph{Gauss-Legendre cubature weight} \cite{stroud,stroud2,trefethen2017,Trefethen2019}.
Consequently, for any polynomial $Q \in \Pi_{m,2n+1}$ of degree $2n+1$ the following cubature rule applies:
\begin{equation}\label{eq:Gauss}
    \int \limits_{\Omega} Q(x)d\Omega = \sum_{\alpha \in A_{m,n}} w_{\alpha}Q(p_{\alpha})\,.
\end{equation}
Summarising: Polynomials of degree $2n+1$ can be (numerically) integrated exactly when sampled on the Legendre grid $P_{m,n}$ of order $n+1$. Thanks to $|P_{m,n}| = (n+1)^m \ll (2n+1)^m$ this makes \emph{Gauss-Legendre integration} a very powerful scheme yielding
\begin{equation}\label{eq:L2}
\li <Q_1,Q_2 \re>_{L^2(\Omega)} = \int \limits_{\Omega_m} Q_1(x)Q_2(x)d\Omega_m = \sum_{\alpha \in A_{m,n}} Q_1(p_\alpha)Q_2(p_\alpha) w_\alpha \,,
\end{equation}
for all $Q_1,Q_2 \in \Pi_{m,n}$. In light of this fact, we propose the following definition.
\begin{definition}[Legendre interpolation and $L^2$-projection ] Let $m,n \in \N$, $P_{m,n}$ be the Legendre grid and $L_\alpha$, $\alpha \in A_{m,n}$ be the corresponding Lagrange polynomials from Eq.\eqref{eq:Lag}.
For continuous functions
$f :\bar \Omega \lo \R$ we denote with
\begin{equation}\label{eq:int_op}
  \Ic_{m,n} : C^0(\Omega,\R)\lo \Pi_{m,n}\,, \quad   \Ic_{m,n}(f) = \sum_{\alpha \in A_{m,n}}f(p_{\alpha}) L_{\alpha} \in \Pi_{m,n}
\end{equation}
the interpolation operator. Moreover, we denote with
\begin{equation}
  \pi_{m,n}: \Ltwo \lo \Pi_{m,n}\,, \quad  \pi_{m,n}(f) = \sum_{\alpha \in A_{m,n}}\frac{1}{w_\alpha}\langle f,  L_{\alpha}\rangle_{\ltwo}L_\alpha \in \Pi_{m,n}
\end{equation}
the $L^2$-projection.
\label{def:pro}
\end{definition}
\begin{remark} It is important to note that $\Ic_{m,n}(f) \neq \pi_{m,n}(f)$ in general.  However, both operators are projections that due to Eq.~\eqref{eq:L2} satisfy
\begin{align*}
    \pi_{m,n}(\pi_{m,n}(f)) = \pi_{m,n}(f)\,, \quad   & \Ic_{m,n}(\Ic_{m,n}(f))= \Ic_{m,n}(f)\,,\\
    \Ic_{m,n}(\pi_{m,n}(f))= \Ic_{m,n}(f)\,, \quad & \pi_{m,n}(\Ic_{m,n}(f))= \Ic_{m,n}(f)\,.
\end{align*}
In fact, both concepts can deliver exponential fast approximation rates (truncation errors) in case the considered function $f$ is analytic \cite{Trefethen2019}.
\end{remark}

How differential operators acting on polynomial spaces can be understood due to these concepts is proposed in the next section.

\subsection{Truncated differential and adjoint operators}

Based on Eq.~\eqref{eq:Lag} we derive exact matrix representations of differential operators acting on the polynomial spaces $\Pi_{m,n}$. This allows to extend Eq.~\eqref{eq:L2} and deliver approximates of the Sobolev norms for general functions $f \in H^k(\Omega,\R)$, $k \in \N$.

For $L_\alpha \in \Pi_{m,n}$ from Eq.~\eqref{eq:Lag} and $ 1\leq i \leq m$ the computation of the values $d_{\alpha,\beta} = \partial_{x_i} L_\alpha (p_{\beta})$, $p_\beta\in P_{m,n}$, $\forall\,\beta \in A_{m,n}$ yield the Lagrange expansion
\begin{equation}
    \partial_{x_i} L_\alpha (x) = \sum_{\beta \in A_{m,n}} d_{\alpha,\beta}L_{\beta}(x) \,.
\end{equation}
Consequently, the matrix
\begin{equation}\label{eq:DI}
    D_i = (d_{\alpha,\beta})_{\alpha,\beta \in A_{m,n}} \in \R^{|A_{m,n}|\times |A_{m,n}|}\,,
\end{equation}
represents the finite dimensional truncation of the differential operator
$\partial_{x_i} : C^1(\Omega,\R) \lo C^0(\Omega,\R)$ to the polynomial space $\Pi_{m,n}$ and for $\beta \in \N^m$ we set
\begin{equation}\label{eq:Dbeta}
    \D_{\beta} = \prod_{j=1}^m D_{\beta_i}\,, \quad \text{with}\,\,\,  D_0 = \I\,,
\end{equation}
to be the approximation of the differential operator $\partial_\beta:=\partial_{x_1}^{\beta_1}\ldots\partial_{x_m}^{\beta_m}$.

For representing the truncation of general adjoint operators we
we consider  the Legendre grid $P_{m,n} = \{p_\alpha : \alpha \in A_{m,n}\}$, $m,n,\in\N$ the positive, symmetric Gauss-Legendre cubature weight matrix  $\W_{m,n} = \mathop{diag}(w_{\alpha})_{\alpha \in A_{m,n}}$, and the evaluation vector $\mathfrak{f} = (f(P_{\alpha}))_{\alpha \in A_{m,n}}\in \R^{|A_{m,n}|}$ for a given function $f : \Omega \lo \R$. With these ingredients we state:

\begin{proposition}\label{prop:adjoint}
    Let  $D_\beta: \Ltwo \lo \Ltwo$, $\beta \in \N^m$  be a differential operator and
    $\D_\beta : \Pi_{m,n}(\Omega) \lo \Pi_{m,n}(\Omega)$ be its truncation to the polynomial space.
    Then the matrix representation of the truncated adjoint operator $\D_\beta^*: \Pi_{m,n}(\Omega) \lo \Pi_{m,n}(\Omega)$
    is given by:
\begin{equation}\label{eq:adjoint}
        \D_\beta^* = \W_{m,n}^{-1} \D_\beta^\top \W_{m,n}\,.
\end{equation}

\end{proposition}
\begin{proof}

We derive Eq.~\eqref{eq:adjoint} due to the Gauss-cubature in terms of Eq.~\eqref{eq:L2}.
Let $Q_1,Q_2\in \Pi_{m,n}$, and denote with $\mathfrak{q}_1=(Q_1(p_\alpha))_{\alpha \in A_{m,n}},
\mathfrak{q}_2= (Q_2(p_\alpha))_{\alpha \in A_{m,n}}  \in \R^{|A_{m,n}|}$ the corresponding evaluation vectors. Then we compute
 \begin{align*}
\langle D_\beta Q_1, Q_2 \rangle_\Ltwo &= \langle  \D \mathfrak{q}_1, \W_{m,n} \mathfrak{q}_2 \rangle
                 = \mathfrak{q}_1^\top \D_\beta^\top \W_{m,n} \mathfrak{q}_2
                 =\mathfrak{q}_1^\top \W_{m,n} \W_{m,n}^{-1} \D_\beta^\top \W_{m,n} \mathfrak{q}_2\\
                &= \langle \W_{m,n}^\top \mathfrak{q}_1, \D_\beta^* \mathfrak{q}_2 \rangle =   \langle  \mathfrak{q}_1, \W_{m,n} \D_\beta^* \mathfrak{q}_2 \rangle
                = \langle Q_1, D_\beta^*Q_2 \rangle_\Ltwo \,,
\end{align*}
proving the statement.
\end{proof}

We provide a matrix representation of the truncation of the adjoint operator $j^*:\Hk \lo \Ltwo$ of the embedding $j :\Hk \lo \Ltwo$.
\begin{theorem}\label{theo:Jstar} Let $j ^* : \Ltwo \lo \Hk$ be the adjoint operator of the embedding $j : \Hk \lo \Ltwo$.
Denote with $\D_\beta$ the representations of the derivatives from Eq.~\eqref{eq:Dbeta} then its truncation
$J^* : \Pi_{m,n}(\Omega) \subseteq \Ltwo \lo \Pi_{m,n}(\Omega) \subseteq \Hk
$
can be represented by the matrix $\J^* \in \R^{|A_{m,n}| \times |A_{m,n}|}$
given by
    \begin{equation}\label{eq:Jstar}
        \J^* = \Big(\sum_\muik{\beta} \D_\beta^* \D_\beta\Big)^{-1}\,.
    \end{equation}
\end{theorem}
\begin{proof} Let $Q_1, Q_2 \in \Pi_{m,n}$,  $P_{m,n}$ the Legendre grid and  $\mathfrak{q}_1=(Q_1(p_\alpha))_{\alpha \in A_{m,n}},
\mathfrak{q}_2= (Q_2(p_\alpha))_{\alpha \in A_{m,n}}  \in \R^{|A_{m,n}|}$ the evaluation vectors,respectively. Then we compute
    \begin{align*}
            \langle Q_1, Q_2 \rangle_\hk &= \sum_\muik{\beta} \langle D_\beta Q_1, D_\beta Q_2 \rangle_\Ltwo
            = \sum_\muik{\beta} \langle D_\beta^* D_\beta Q_1,Q_2\rangle_\Ltwo\\
            &= \langle \big (\sum_\muik{\beta}  D_\beta^* D_\beta\big) Q_1, Q_2 \rangle_\Ltwo\,.
    \end{align*}
Thus, setting $J^{* -1}:= \sum_\muik{\beta}  D_\beta^* D_\beta$ yields that due to the identity above $J^{*-1}$ is a symmetric and positive definite linear operator on a finite dimensional space implying its invertibility.
Due to
\begin{equation*}
\langle \big (\sum_\muik{\beta}  D_\beta^* D_\beta\big) Q_1, Q_2 \rangle_\Ltwo=  \langle \big(\sum_\muik{\beta}\D_\beta^* \D_\beta\big) \mathfrak{q}_1,\mathfrak{q}_2 \rangle
\end{equation*}
we realise that
$\J^{*-1}:=\sum_\muik{\beta}\D_\beta^* \D_\beta$ represents $J^{*-1}$.
\end{proof}

As introduced, the PSMs rely on the Chebyshev polynomials $\{T_{\alpha}\}_{\alpha \in A_{m,n}}$, $m,n\in\N$, Eq.~\eqref{eq:ChebP}. For later purpose we provide the basis transformation between the $T_{\alpha}$ and the  Lagrange basis $L_\alpha$ in the Legendre grid $P_{m,n}$. That is to consider the matrix
\begin{equation}\label{eq:chebTR}
    \T = (T_{\beta}(p_\alpha))_{\alpha,\beta \in A_{m,n}} \in \R^{|A_{m,n}|\times|A_{m,n}|} \quad \text{and its inverse}\quad \T^{-1} \in \R^{|A_{m,n}|\times|A_{m,n}|}\,.
\end{equation}
Given Lagrange coefficients $C = (c_{\alpha})_{\alpha \in A_{m,n}}$ of a polynomial $Q= \sum_{\alpha \in A_{m,n}}c_\alpha L_\alpha$, $\Theta=(\theta_\alpha)_{\alpha \in A_{m,n}}=\T^{-1}C$ yields the  coefficients of its Chebyshev representation $Q=\sum_{\alpha \in A_{m,n}}\theta_\alpha T_\alpha $. Vice versa $D=(d_\alpha)_{\alpha \in A_{m,n}}=\T\Theta$ yields the Lagrange coefficients of its Chebyshev expansion.
We close this section, by deriving a matrix representation of the trace operator, Eq.~\eqref{eq:tr}:
\begin{definition}[Truncated trace operator]\label{def:TR} Let
$\mathrm{tr}: H^k(\Omega,\R) \lo L^2(\partial \Omega,\R)$
be the trace operator, Eq.~\eqref{eq:tr}.
Denote with $P_{m-1,n,j}^\pm \subseteq\partial\Omega_{j}^{\pm}$  the $\mbox{m-1}$-dimensional Legendre grids for each of the faces  $\partial\Omega_{j}^{\pm} =\{x \in \Omega : x_j = \pm 1\}$ of the hypercube $\Omega$.
Then the matrix  $\mathbb{S}_{m,n,j}^\pm\in \in \R^{|A_{m-1,n}|\times |A_{m,n}|}$ with
\begin{equation}\label{eq:TR}
      \mathbb{S}_{m,n,j}^\pm=  (T_{\alpha}(p_\gamma))_{(\gamma,\alpha) \in A_{m-1,n}\times A_{m,n}}\,, \quad p_\gamma \in P_{m-1,n,j}^\pm\,, j=1,\dots,m\,.
\end{equation}
represents the truncated trace operator
$\mathrm{tr}:\Pi_{m,n} \lo \Pi_{m-1,n}(\partial\Omega^\pm_j)$ for each of the faces $\partial\Omega^\pm_j$.
\end{definition}

The derived representations of the truncated differential and adjoint operators enable to derive cubature rules for the truncated Sobolev spaces.

\subsection{Sobolev cubatures}
Based on the classic Gauss-Legendre cubature Eq.~\eqref{eq:L2} we, here, derive general \emph{Sobolev cubatures}.
We start by defining:

\begin{definition}[Truncated (dual) inner product and norm] \label{def:dual}For $\beta \in \N^m$, $\|\beta\|_1 \leq k$, $m,n \in \N$ we consider the truncated differential operator $D_\beta$ and its adjoint
$D_\beta : \Pi_{m,n}(\Omega) \lo \Pi_{m,n}(\Omega)$, $ D_\beta^* : \Pi_{m,n}(\Omega) \lo \Pi_{m,n}(\Omega)$ satisfying
\begin{equation*}
  \langle D_\beta Q_1, Q_2\rangle_\ltwo = \langle  Q_1, D_\beta^* Q_2\rangle_\ltwo\,,\quad \forall Q_1,Q_2 \in \Pi_{m,n}
\end{equation*}
Given the matrix representations $\D_\beta$, $\D_\beta^*= W_{m,n}^{-1}\D_{\beta}^T \W_{m,n}
$ from Proposition~\ref{prop:adjoint}, $\J^*$ from Eq.~\eqref{eq:Jstar} and its formal dual
\begin{equation*}
    \J^* = \Big(\sum_\muik{\beta} \D_\beta^* \D_\beta\Big)^{-1}\,, \quad \underline{\J}^* = \Big(\sum_\muik{\beta} \D_\beta \D_\beta^*\Big)^{-1}\,,
\end{equation*}
we introduce
\begin{equation*}
\W_{m,n,k} = \W_{m,n} {\J^*}^{-1}\,, \W_{m,n,-k} = \W_{m,n} \J^*\,, \quad
 \underline{\W}_{m,n,k} = \W_{m,n} \underline{\J}^{*-1}\,,
 \underline{\W}_{m,n,-k} =\W_{m,n} \underline{\J}^*\,,
\end{equation*}
and for $f,g \in \Pi_{m,n}$ and their dual distributions $F = \langle f, \cdot\rangle_\ltwo$, $G = \langle g, \cdot\rangle_\ltwo$ we set
\begin{alignat}{5}
    &\langle f,g\rangle_{\hk} &=& \sum_{\beta \in \N^m, \|\beta\|_1\leq k} \langle D_{\beta}f, D_{\beta}g\rangle_\ltwo &=& \langle \mathfrak{f}, \W_{m,n,k}  \mathfrak{g}  \rangle & \nonumber \\
 &\langle f,g\rangle_{\hk,*} &=& \sum_{\beta \in \N^m, \|\beta\|_1\leq k} \langle D_{\beta}^*f, D_{\beta}^*g\rangle_\ltwo &=& \langle \mathfrak{f}, \underline{\W}_{m,n,k}  \mathfrak{g}  \rangle &  \nonumber\\
        &\langle F,G\rangle_{\hmk} &=& \sum_{\beta \in \N^m, \|\beta\|_1\leq k} \langle D_{\beta} J^*f, D_{\beta} J^*g\rangle_\ltwo  &=& \langle \mathfrak{f}, \W_{m,n,-k}  \mathfrak{g}  \rangle & \nonumber \\
        &\langle F,G\rangle_{\hmk,*} &=& \sum_{\beta \in \N^m, \|\beta\|_1\leq k} \langle D_{\beta}^* J^*f, D_{\beta}^* J^*g\rangle_\ltwo  &=& \langle \mathfrak{f}, \underline{\W}_{m,n,-k}  \mathfrak{g}  \rangle\,, & \label{eq:SC}
\end{alignat}
where $\mathfrak{f} = (f(p_{\alpha}))_{\alpha \in A_{m,n}} \in \R^{|A_{m,n}|}$, $\mathfrak{g} = (g(p_{\alpha}))_{\alpha \in A_{m,n}} \in \R^{|A_{m,n}|}$ are the evaluation vectors of  $f,g$ in the Legendre nodes $p_\alpha \in P_{m,n}$, respectively.
The corresponding norms are given by
\begin{alignat}{4}
\|f\|_{\hk} &=  \langle f,f\rangle_{\hk}^{1/2}\,, & \|f\|_{\hk,*} &=  \langle f,f\rangle_{\hk,*}^{1/2}\nonumber\\
\|F\|_{\hmk}&= \langle F,F\rangle_{\hmk}^{1/2}\,, & \quad \|F\|_{\hmk,*} &=  \langle F,F\rangle_{\hmk,*}^{1/2}\,.
\end{alignat}
\end{definition}
In fact, while including the $L^2$-inner product for $\beta =0$, the expressions above define inner products and norms.
We deduce the exactness of the equations.

\begin{theorem}[Sobolev cubatures]\label{theo:CUB}
Let $f, g \in \Hk$ and $F = \langle f,\cdot\rangle, G = \langle g,\cdot\rangle \in \Hmk$.
Then
the approximations
given by Definition~\ref{def:dual}, Eq.~\eqref{eq:SC},
are exact for all $f,g \in \Pi_{m,n}$.
\end{theorem}
\begin{proof} By combining Proposition~\ref{prop:Hmk}, Theorem~\ref{theo:Jstar} and  $\Ic_{m,n}(\pi_{m,n}(f)) = \pi_{m,n}(f)$ the proof follows.

\end{proof}

The following observation is helpful for computing the Sobolev cubatures.

\begin{corollary}\label{cor:test} Let  $f \in \Pi_{m,n}$ and the assumptions of
Definition~\ref{def:dual} be fulfilled. Then the following identities hold:
\begin{align}
    \langle D_\beta f , D_\beta f \rangle_\Ltwo &=  \sum_{\alpha \in A_{m,n}} \frac{1}{w_\alpha} \langle D_\beta f,  L_\beta \rangle_\Ltwo^2 \nonumber\\
    \langle D^*_\beta f , D^*_\beta f \rangle_\Ltwo &= \sum_{\alpha \in A_{m,n}} \frac{1}{w_\alpha} \langle  f,  D_\beta L_\alpha \rangle_\Ltwo^2 \label{eq:dual_L}
\end{align}
\end{corollary}
\begin{proof}
We use Proposition~\ref{prop:adjoint} in terms of  $\D^*_\beta = \W_{m,n}^{-1}\D_{\beta}^T\W_{m,n}$ and due to Theorem~\ref{theo:CUB} compute
    \begin{align*}
            \langle D^*_\beta f , D^*_\beta f \rangle_\Ltwo & = \langle \D_\beta^*\mathfrak{f}, \W_{m,n}\D_\beta^*\mathfrak{f}  \rangle = \langle \W_{m,n}^{-1}\D_\beta^\top \W_{m,n} \mathfrak{f}, \D_\beta^\top \W_{m,n}\mathfrak{f}  \rangle \\
            &= \sum_{\alpha \in A_{m,n}} \frac{1}{w_\alpha} \langle \mathfrak{f},   \D_\beta^\top \W_{m,n} e_\alpha \rangle^2  = \sum_{\alpha \in A_{m,n}} \frac{1}{w_\alpha} \langle f, D_\beta L_\alpha \rangle_\Ltwo^2\,,
    \end{align*}
    where $e_\alpha$ is the $\alpha$-th standard basis vector of $\R^{|A_{m,n}|}$.
The analog computation applies for $D_\beta$.
\end{proof}

In fact, when considering the truncated (dual) norms ($\|\cdot\|_{\hmk,*}$, $\|\cdot\|_{\hk,*}$), $\|\cdot\|_{\hmk}$, $\|\cdot\|_{\hk}$, computations based on Eq.~\eqref{eq:dual_L} are straightforwardly achieved and documented in \cite{PDE_repo}. We provide the formal setup next.

\section{PDE formulations}
In light of the provided perspectives, we follow \cite{jost2002,brezis2011} to propose the following formalization of classic PDE problems. For the sake of simplicity, we focus on  classic \emph{Poisson type equations}. Extensions to more general PDE problems can be derived  once the notion is given, see Section~\ref{sec:Num}.

\subsection{Poisson equation}
Let us consider the Poisson equation, for $f \in C^0(\Omega,\R)$. The \emph{strong Poisson problem} with Dirichlet boundary condition $g \in C^0(\partial \Omega,\R)$ seeks for solutions $u\in C^2(\Omega,\mathbb{R})$ fulfilling:
\begin{equation}\label{eq:str_P}
   \li\{ \begin{array}{rll}
       -\Delta u(x) -f(x) &= 0  &,  \forall x\in\Omega  \\
         u(x)  -g(x)      &= 0  &,  \forall x\in\partial\Omega\,.
\end{array}\re.
\end{equation}
By using the notion of weak derivatives we can formulate a \emph{weaker version of the Poisson equation}. That is, finding  $u\in H^2(\Omega,\mathbb{R})\subseteq C^0(\Omega,\R)$ fulfilling
\begin{equation}\label{eq:var_form}
    \int\limits_{\Omega}(-\Delta u-f)\phi\textrm{ d}x, \textrm{  }\forall \phi \in C^{\infty}(\Omega,\mathbb{R}),
\end{equation}
subjected to the same Dirichlet boundary conditions as in equation \eqref{eq:str_P}. The notions give rise to the following  optimisation problems.

\subsection{PDE loss}\label{sec:PDE_L}

We use the Sobolev space setting $H^k(\Omega,\R)$, $H^l(\partial\Omega,\R)$, $k,l \in \Z$ for introducing soft-constrained PDE-losses that impose the Poisson-PDE-solution with general boundary condition as one \emph{global variational optimisation problem}.
\begin{definition}\label{def:loss}
    Given the setup of Eq.~\eqref{eq:str_P} the  strong PDE-loss $\Lc_{\mathrm{strong}} : H^{k+2}(\Omega,\R)\cap H^l(\partial\Omega,\R)\lo \R$, $k, l \in\N$  is defined  by
\begin{equation}\label{eq:sL}
    \Lc_{\mathrm{strong}}(u) = r_{\mathrm{strong}}(u) + s_{\mathrm{strong}}(u)
   = \|-\Delta u -f \|^2_{\hk} + \| u_{| \partial \Omega} -g \|^2_{\ltwo}\,.
\end{equation}
The weak PDE-loss $\Lc_{\mathrm{weak}} : H^{k+2}(\Omega,\R)\cap H^l(\partial\Omega,\R)\lo \R$, reflecting the
weak formulation in Eq.~\eqref{eq:var_form}, is given by
\begin{alignat}{1}
\Lc_{\mathrm{weak}}(u) &= r_{\mathrm{weak}}(u) + s_{\mathrm{weak}}(u) \nonumber \\
&=  \sup_{\phi \in C^\infty(\Omega,\R)} \li <-\Delta u-f,\phi \re>^2_{\hk} +  \sup_{\phi \in C^{\infty}(\partial\Omega,\R)} \li <u-g,\phi \re>^2_{\ltwo}\,.  \label{eq:wL}
\end{alignat}
\end{definition}
Truncations of the the strong loss $\Lc_{\mathrm{strong}}:\Pi_{m,n} \lo \R^+$ can be derived by applying the Sobolev cubatures from Definition~\ref{def:dual}. A truncation  $\Lc_{\mathrm{weak}}:\Pi_{m,n} \lo \R^+$
of the weak PDE-loss, Eq.~\eqref{eq:wL} is given by requiring Eq.~\eqref{eq:var_form} to be fulfilled only for all polynomial test functions $\varphi\in \Pi_{m,n} = \mathrm{span}(L_{\alpha})_{\alpha \in A_{m,n}}$ spanned by the Lagrange polynomials. Hence, we consider
\begin{equation}
 r_{\mathrm{weak}}(u) \approx
\sum_{\alpha\in A_{m,n}} \li <-\Delta u-f,L_\alpha \re>^2_{H^k(\Omega)}\,, \quad
s_{\mathrm{weak}}(u) \approx \sum_{\alpha\in A_{m,n}} \li <u-g,L_\alpha \re>^2_{H^{l}(\Omega)}\,. \label{eq:awL}
\end{equation}

While Definition~\ref{def:loss} includes the case $k,l <0$ the corresponding losses occur when replacing $\|\cdot\|_{\hk}, \|\cdot\|_{\hmk}$ with $\|\cdot\|_{\hk}$, $\|\cdot\|_{\hmk,*}$, yielding well-defined notions due to Proposition~\ref{prop:Hmk}. Next, we derive the corresponding gradient flows of the given losses.

\subsection{Variational gradient flows}
Given  a polynomial $Q_{C_0} = \sum_{\alpha \in A_{m,n}}c_\alpha L_\alpha$ in Lagrange expansion with respect to the Legendre grid $P_{m,n} \subseteq \Omega$ with  coefficients $C_0 =(c_\alpha)_{\alpha\in A_{m,n}}\in\R^{|A_{m,n}|}$. We consider the truncated loss
\linebreak
$\Lc : \R^{|A_{m,n}|} \lo \R^+$, $\Lc=\Lc[C]$ acting on the coefficients and the \emph{gradient flow ODE}
\begin{equation}\label{eq:ODE}
 \partial_t C(t) = - \nabla \Lc(Q_{C(t)})      \quad  \,, C(0)= C_0\,.
\end{equation}
Combining the identity $Q_C(p_{\alpha}) =c_\alpha$, with Definition~\ref{def:dual} for the evaluation vector
$\mathfrak{f}=(f(p_{\alpha}))_{\alpha\in A_{m,n}}$ we derive the following expression for the $L^2$-gradient in case for the strong loss $\Lc =\Lc_{\mathrm{strong}}$ from Eq.~\eqref{eq:sL},i.e,
\begin{align*}
    \nabla_C(r_{\textrm{strong}}) &= \nabla_C\langle\big((\D_{x_1}^2+\cdots+\D_{x_m}^2)C+\mathfrak{f}\big), W_{m,n}\big((\D_{x_1}^2+\cdots+\D_{x_m}^2)C+\mathfrak{f}\big)\rangle\,,
\end{align*}
where according to Eq.~\eqref{eq:Dbeta}, $\D_{x_i}^2 = \D_{2e_i}$ with $e_i \in \R^m$ being the standard basis, $i=1,\dots,m$. Thus,
\begin{align}
\nabla_C(r_{\textrm{strong}}) &= -2(\D_{x_1}^2+ \cdots + \D_{x_m}^2)^T\W_{m,n}\big((\D_{x_1}^2+\cdots+\D_{x_m}^2)C+\mathfrak{f}\big)\,, \label{eq:gradS}\\
 \nabla_C( s_{\textrm{strong}})^\pm_j &= 2\W_{m-1,n}(\mathbb{S}_{m,n,j}^\pm C-\mathfrak{g^\pm_j})\,, \quad j =1,\dots,m\,,\nonumber
\end{align}
where $\mathfrak{g^\pm_j}$ is the evaluation vector of $g$ in the $\mbox{m-1}$-dimensional Legendre grid $P_{m-1,n,j}^\pm \subseteq \partial \Omega_j^\pm$ contained in each face  $\partial \Omega_j^\pm$ of $\Omega$, and $\mathbb{S}_{m,n,j}^\pm$ denotes the truncated trace operator, Definition~\ref{def:TR}.

Analogously, in case of the weak loss
$\Lc =\Lc_{\mathrm{weak}}$ from Eq.~\eqref{eq:wL} we derive
\begin{align}
\nabla_C(r_{\textrm{weak}}) &= -2(\D_{x_1}^2+ \cdots \D_{x_m}^2)^T\W_{m,n}^2\big((\D_{x_1}^2+\cdots+\D_{x_m}^2)C+\mathfrak{f}\big) \label{eq:gradW}\\
 \nabla_C( s_{\textrm{weak}})^\pm_j &= 2\W_{m-1,n}^2(\mathbb{S}_{m,n,j}^\pm C-\mathfrak{g^\pm_j})\,.\nonumber
\end{align}
Formulas for choosing truncated dual norms $\|\cdot\|_\hk$, $\|\cdot\|_{\hk,*}$, $0<k<\infty$ as in Definition~\ref{def:dual} result when replacing
$\W_{m,n}$ with the corresponding cubature matrix, e.g.  $\W_{m,n}\J^{*-1}$,
from Definition~\ref{def:dual} in Eq.~\eqref{eq:gradS}, while  in Eq.~\eqref{eq:gradW} $\W_{m,n}^2\J^{*-1}$ occurs.

For all cases, Corollary~\ref{cor:test} provides the baseline for numerical stable implementations, which are realised and documented in \cite{PDE_repo}.

\subsubsection{Analytic variation of linear PDEs}\label{sec:AD}

Given the analytic expressions of the variational gradients in Eq.~\eqref{eq:gradS},\eqref{eq:gradW} we derive the analytic solution of the gradient descent, Eq.~\eqref{eq:ODE}:
To do so, we shorten $\D:= (\D_{x_1}^2+\cdots+\D_{x_m}^2)$,  $\D^*:=\D^T\mathbb{W}_{m,n}$, $\mathbb{S}:= \sum_{j=1}^m S_{m,n,j}^\pm$, $\mathbb{S}^*\mathfrak{g}:=\mathbb{W}_{m-1,n}\sum_{j=1}^m \mathfrak{g}^\pm_j$
and realise that  Eq.~\eqref{eq:ODE} becomes:
\begin{equation*}
\frac{d}{dt}C(t)  = -2(\D^* \D +\mathbb{S}^*\mathbb{S}) C(t)+2(\mathbb{S}^*\mathfrak{g}-\D^*\mathfrak{f})\,.
\end{equation*}
By applying the variation of parameters we derive the solution of the ODE as:
\begin{equation*}
C(t) = \mathrm{exp}(-t\cdot\K^*\K)C_0+2(\mathbb{I}-\mathrm{exp}(-t\cdot\K^*\K))(\K^*\K)^{+}(\mathbb{S}^*\mathfrak{g}-\D^*\mathfrak{f})\,,
\end{equation*}
where $\K^*\K := 2(\D^*\D + \mathbb{S}^*\mathbb{S})$, and $(\K^*\K)^{+}$ denotes the \emph{Moore--Penrose pseudo-left-inverse}, see e.g., \cite{ben2003,Lloyd_Num}. In case, where $\K^*\K$ is a positive definite matrix that imples
\begin{equation}\label{eq:AS}
C_\infty:=\lim_{t\rightarrow \infty}C(t)= (\K^*\K)^{-1}(\mathbb{S}^*\mathfrak{g}-\D^*\mathfrak{f})\,.
\end{equation}
While we expect that $\K^*\K$ is positive definite, and thus invertible, whenever the underlying PDE problem is well posed and posses a unique solution a formal proof of this implication requires a deeper theoretical study that is out of scope of this article. Empirical demonstrations in Section~\ref{sec:Num}, however, suggest this expectation to be genuine.

Whatsoever, non-linear PDEs or inverse PDE problems can not be solved due to Eq.~\eqref{eq:AS} and require gradient descent methods, realising Eq.~\eqref{eq:ODE}. A deeper investigation of such approaches is given in the next section.

\subsection{Exponential convergence of  $\lambda$-convex gradient flows}\label{sec:AED}

In practice more general problems than  linear (forward) PDE problems occur.
We motivate this section by considering an inverse problem for the Poisson equation \eqref{eq:str_P}. That is to consider a function $f: \Omega \lo \R$ and an unknown parameter
$\mu \in \R$ and pose the PDE problem
\begin{equation}\label{eq:inv}
   \li\{ \begin{array}{rll}
       -\Delta u(x) -\mu f(x)&= 0  &,  \forall x\in\Omega  \\
         u(x)  -g(x)      &= 0  &,  \forall  x\in\Omega
\end{array}\re.
\end{equation}
where $g$ is one specific Poisson solution, i.e., $\Delta g =  \mu f$ on $\Omega$. For inferring the parameter $\mu\in \mathbb{R}$ and the PDE solutions simultaneously  we assume that $g$ can be sampled at the Legendre grid $P_{m,n}$ and formulate the truncated (polynomial) loss by:
\begin{equation}\label{eq:IP_ref}
    \mathcal{L}[C,\mu] = \|-\Delta Q_C -\mu f\|^2_{H^k(\Omega)} +  \|Q_C -g\|^2_{H^l(\Omega)}\,, \quad k,l \in \N\,.
\end{equation}
While the PDE solution depends on $\mu$ itself, we cannot compute  the analytic solution directly. Instead, we apply an iterative gradient descent for deriving the solution based on Eq.~\eqref{eq:IP_ref}. We prove that the proposed approach converges  exponentially fast for even more general problems.
\begin{definition}
A differentiable functional $\mathcal{F}: \mathbb{R}^{|A_{m,n}|}\rightarrow \mathbb{R}$ is called $\lambda$-convex if there is a $\lambda >0$ such that:
\begin{equation}\label{eq:strong_conv}
    \mathcal{F}[x]\geq \mathcal{F}[y]+\nabla \mathcal{F}[y]^T(x-y)+\frac{\lambda}{2}\|x-y\|^2, \forall x,y\in \mathbb{R}^{|A|}
\end{equation}
\end{definition}
\begin{theorem}\label{thm_conv} Given a truncated loss
$\Lc : \R^{|A_{m,n}|}\lo \R^+$, $m,n \in \N$, as in Section~\ref{sec:PDE_L}, that is $\lambda$-convex and differentiable and
assume that the optimal solution $C_\infty := \mathrm{argmin}_{C\in \mathbb{R}^{|A_{m,n}|}} \Lc[C]$ minimizing the variational problem
exists and is unique. Then both the loss and the gradient descent
\begin{equation*}
 \partial_t C(t) = - \nabla \Lc(Q_{C(t)})      \quad  \,, C(0)= C_0\,.
\end{equation*}
converge exponentially fast as $t\rightarrow \infty$:
\begin{equation}
    \frac{\lambda}{2}\|C(t)-C_\infty \|^2 \leq \Lc[C(t)]-\Lc[C_\infty]\leq e^{-2\lambda t}(\Lc[C_0]-\Lc[C_\infty]).
\end{equation}
\end{theorem}
\begin{proof}
The proof of the statement is given in the appendix.
\end{proof}

We give some insights to assert in which situations Theorem~\ref{thm_conv} applies:
\begin{proposition}\label{prop:LA} Let $A\in \mathbb{R}^{r\times s}$, $r \geq s \in \N$ be a positive definite matrix, $\lambda>0$ be the smallest eigenvalue of $A$ then the affine loss
\begin{equation}\label{eq:AL}
    \Lc(C) = \|AC+b\|^2\,, \quad b \in \R^{r}
\end{equation}
is $\lambda$-convex.
\end{proposition}
\begin{proof}
We start by observing that any norm is $1-$convex, in particular it holds:
\begin{equation}
\|x\|^2 = \|y\|^2+(\nabla \|y\|^2)^T(x-y)+\|x-y\|^2\,,
\end{equation}
where $(\nabla\|y\|^2)^T(x-y) = 2\langle y,x-y\rangle$.

By replacing the roles of $x,y$ with $Ax +b$, $Ay+b$, respectively, we compute:
\begin{align*}
    \|Ax + b\|^2 &= \|Ay+b\|^2+2\langle Ay+b,A(x-y)\rangle +\|A(x-y)\|^2
    \\&=\|Ay+b\|^2+2\langle A^T(Ay+b),x-y\rangle+\|A(x-y)\|^2\\
    &\geq \|Ay+b\|^2+2(\nabla(\|Ay+b\|^2),x-y)+\lambda\|x-y\|^2\,,
\end{align*}
where $\nabla(\|Ay+b\|^2) = 2(A^T(Ay+b))$.
\end{proof}
We want to note that the assumption on $A$ in Proposition~\ref{prop:LA} can be relaxed:
\begin{remark}[Exponential convergence of non-unique solutions] Given that $\ker A \neq 0$, but  $b\in \R^r$ in Eq.~\eqref{eq:AL} satisfies $b \in \mathrm{coker} A^T =\{x \in \R^s : A^Tx \neq 0 \}$ we observe that solving $AC =b$ is equivalent to minimising
\begin{equation}\label{eq:WP}
   \Lc(C)= \| A^TA C+ A^Tb\|^2 =  \| A' C+ b'\|^2 \,,
\end{equation}
with $b' = A^Tb$, $A' = A^TA$. Let $\lambda >0$ be the smallest non-vanishing eigenvalue of $A'=A^TA$. While $\mathrm{coker} A^T \cong  \im A$, $\mathcal{L}$ is $\lambda$-convex on $(\ker A)^\perp$. Due to Theorem~\ref{thm_conv} and Proposition~\ref{prop:LA} this implies that the gradient descent of well-posed problems, Eq.~\eqref{eq:WP}, converges exponentially fast to a solution as long as the initial coefficients $C_0=C(0) \not\in \ker A$ were proper chosen.
\end{remark}
The practical relevance of the observation above is part of the empirical demonstrations
of our proposed concepts given in the next section.

 \section{Numerical experiments}
 \label{sec:Num}
We designed several numerical experiments for validating our theoretical results.
The computations of the PSMs were executed on a standard Linux laptop (Intel(R) Core(TM) i7-1065G7 CPU @ 1.30GHz, 32\,GB RAM).
Precomputation of the Sobolev cubature matrices is realised as a feature of the open source package  \cite{minterpy}. The PSMs are realised by Chebyshev polynomials, Eq.~\eqref{eq:ChebP}, constrained on Legendre grids as asserted in Eq.~\eqref{eq:chebTR}. All PINN experiments were executed on the NVIDIA V100 cluster at HZDR.
Complete code and benchmark sets is available at \cite{PDE_repo}. We intensively compared several PINN approaches in our previous work \cite{cardona2022replacing}. That is why, apart from classic PINNs, here, we focus on comparing our approach with the PINN-methods that turned out to be most reliable:
\begin{enumerate}
    \item[i)]  \emph{Classic PINNs} with the strong $L^2$-MSE loss based on \cite{RAISSI2019686}, as described in the introduction.
    \item[ii)]  \emph{Inverse Dirichlet Balancing (ID-PINNs)} with the $L^2$-MSE loss \cite{maddu2021}, as described in the introduction.
    \item[iii)] \emph{Sobolev Cubature PINNs (SC-PINNs)} \cite{cardona2022replacing}, with the weak $L^2$-loss for all the experiments unless specified otherwise.
    \item[iv)] \emph{Gradient flow optimised PSMs (GF-PSM)}, using the LBFGS-optimiser \cite{Byrd} for the forward problem with the $H^{-1}_\star$-norm for the PDE loss and the strong $L^2-$loss for the other terms (unless further specified). Poisson and QHO Inverse problems are solved by an Implicit-Euler time integration \cite{inbook} with the strong $L^2$ loss and Newton-Raphson \cite{Newton} for the Navier Stokes inverse problem, with the $H^{-1}_\star$ loss.
    \item[iv)] \emph{Analytic Descent (AD-PSM)}, deriving the PSM by the analytic descent given in Eq.~\eqref{eq:AS} by choosing the dual $H^{-1}_\star$-loss, Eq.~\eqref{eq:SC}, for the PDE-loss and the strong $L^2$-loss for the remaining terms.
\end{enumerate}
For measuring the approximation errors of a ground truth function $g : \Omega \lo \R$ by a surrogate model $u$
we evaluate
both on equidistant grids $\mathfrak{g} = (g(p_i))_{i=1,\dots,N}\in \R^N$  $\mathfrak{u} (u(p_i))_{i=1,\dots,N}\in\R^N$ of size $N$ and compute the  $l_1,l_\infty$-errors $\epsilon_1 := \|\mathfrak{g}-\mathfrak{u}\|_1 /N$, $\epsilon_\infty :=\|\mathfrak{g}-\mathfrak{u}\|_\infty$.
We used $N = 100^2$ points for the 2D problems and $N=20^4$ points for the 4D problem.
The parameter inference error is denoted with
$\epsilon_\mu :=|\mu -\mu_{gt}|$.

All  models are trained with the same number of training points $T$. For the PINN and ID-PINN methods, the training points are given by randomly sampling from an equidistant grid $G$ of size $|G|\gg N$. For the SC-PINN and the PSM methods the training points are given by the Legendre grids.  CPU-training-runtimes are reported in seconds.

\subsection{2D and 4D Poisson equations}\label{sec:P}
We start by considering the Poisson problem in dimension $m=2$ in the strong formulation with Dirichlet boundary conditions, Eq.~\eqref{eq:str_P}.

\newcommand{\tabHT}{
\begin{tabular}[t]{llll}
  & \multicolumn{2}{c}{Approximation error} & Runtime (s)  \\
$\dim=2$ & $\epsilon_1$  & $\epsilon_\infty$  &  \\
\hline
PINN & $4.43\cdot10^{-3}$ & $5.2\cdot 10^{-2}$ & $t = 886$\\
ID-PINN & $5.23\cdot10^{-3}$ & $1.9\cdot 10^{-2}$ & $t = 1356$\\
SC-PINN & $2.52\cdot10^{-3}$ & $3.33\cdot 10^{-2}$ & $t = 79.2$\\
GF-PSM & $5.37\cdot10^{-5}$ & $2.94\cdot10^{-3}$ & $t = {12.84}$\\
AD-PSM & $\bf8.79\cdot10^{-10}$ & $\bf1.25\cdot10^{-8}$ & $\bf t = {1.21}$\\
& & &\\
     & \multicolumn{2}{c}{Approximation error} & Runtime (s)  \\
    $\dim=4$ & $\epsilon_1$  & $\epsilon_\infty$  &  \\
    \hline
     GF-PSM & $1.33\cdot10^{-6}$ & $1.0\cdot10^{-3}$ & $t = {173.59s}$\\
     AD-PSM & $\bf 5.42 \cdot 10^{-8}$ & $\bf 6.37\cdot10^{-7}$ & $t = {\bf 7.66s}$
\end{tabular}
}
\begin{table}[h]
\begin{minipage}[b]{.40\textwidth}%
\centering
\vspace*{0pt}\includegraphics[width=1.0\textwidth]{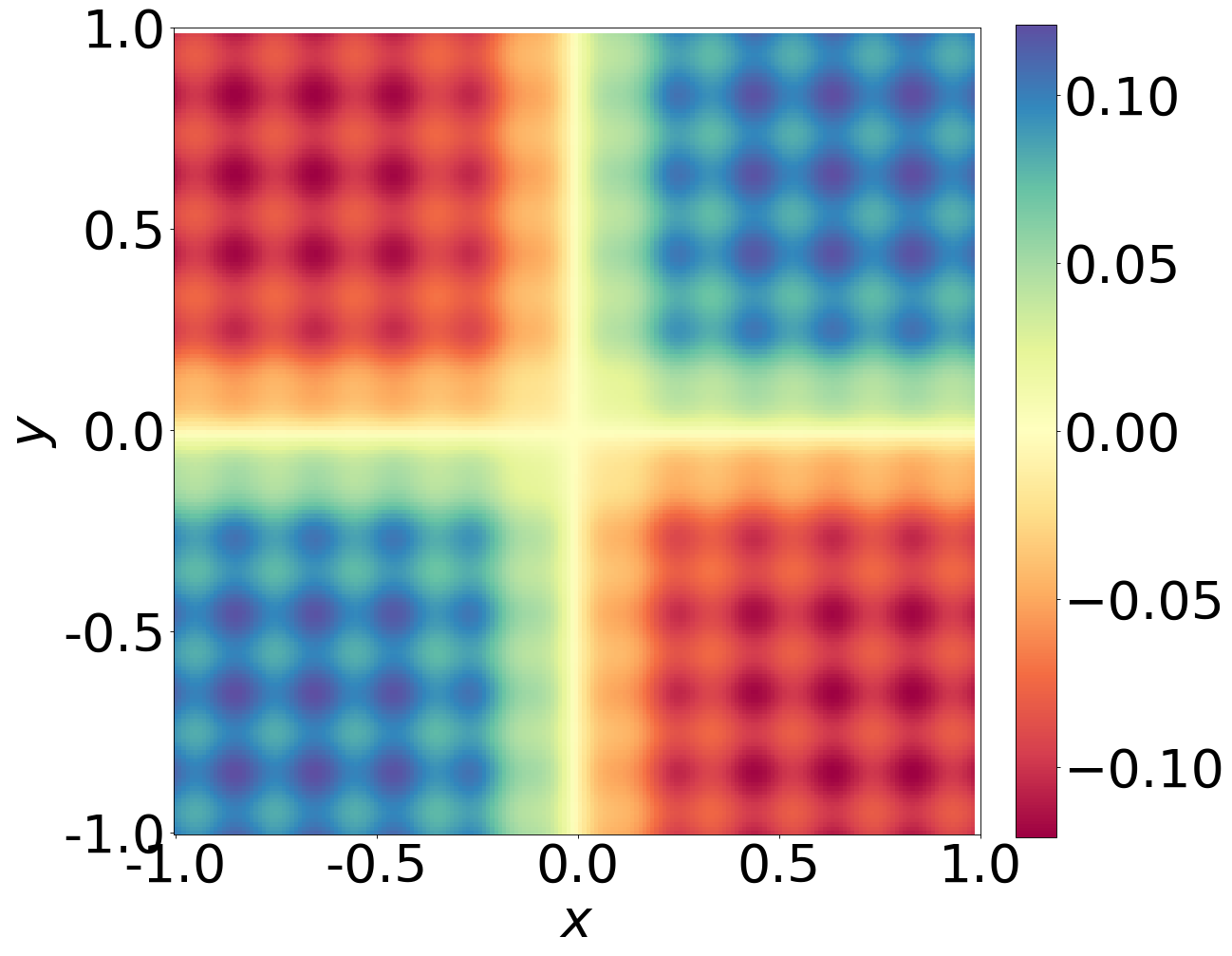}
\captionof{figure}{ Solution for 2D Poisson problem}%
\end{minipage}
\begin{minipage}[b]{.60\textwidth }%
\tabHT
\caption{Errors for 2D and 4D Poisson forward problem}\label{NP_Poisson}%
\end{minipage}%
\end{table}

\begin{experiment}[Non-periodic 2D-Poisson forward problem with hard transitions]
We consider the Poisson equation with right hand side function $f$ given by
\begin{align*}
    f(x,y) = &C (A \sin(\omega y) + \tanh(\beta y))
    (-A \omega^2 \sin(\omega x) - 2 \beta^2 \tanh(\beta x)
    \mathrm{sech}^2(\beta x)) \\
    &+ C (A \sin(\omega x) + \tanh(\beta x))(-A \omega^2\sin(\omega y) - 2 \beta^2 \tanh(\beta y) \mathrm{sech}^2(\beta y)),
    \end{align*}
with $C =0.1, A = 0.1, \beta = 5, \omega = 10\pi$. All the experiments where conducted with the same number of training points, as required for the Sobolev cubatures of degree $n=50$ in the domain and $n=100$ for the boundary. For the SC-PINN the weak $L^2$
-loss was used for the PDE loss and for the boundary.
\end{experiment}

Table~\ref{NP_Poisson} (top) reports the results and shows that the PSM methods outperform all PINN approaches, both, in accuracy and runtime. AD-PSM  reaches seven orders of magnitude smaller $\epsilon_1$-error and requires up to three orders of magnitude less runtime. The GF-PSM performance is non-compatible to AD-PSM, but still far better than the PINN alternatives. The results clearly demonstrate the PSM method to be capable of finding solutions to non-trivial linear PDEs with general non-periodic boundary conditions.

The following experiment indicates that this observation maintains true even for higher dimensional problems.

\begin{experiment}[4D Poisson equation forward problem]
We seek for a solution of a Poisson problem in dimension $m=4$. We choose
$$f(x):= -4 \omega^2 g(x),$$
with $\omega = 1$ and periodic boundary condition $g(x):= \sin(\omega x_1)\cos(\omega x_2)\sin(\omega x_3)\cos(\omega x_4)$
yielding $u(x) =g(x)$ to be the analytic solution. We choose Sobolev cubatures of degree $n=8$ for both, the domain and the boundary loss.
\end{experiment}

In Table~\ref{NP_Poisson}~(bottom) the approximation errors are reported. While all PINN approaches failed to provide any reasonable solution, the PINN-results were skipped.
In contrast, the PSMs can recover the solution accurately. We want to stress that the PSM runtimes are still smaller than the training runtimes of ID-PINN or the standard PINNs occuring for the  analogue  2D Poisson problem, validating again its superior efficiency.

\begin{experiment}[2D Poisson inverse problem]\label{exp:PI}
We consider the inverse 2D-Poisson problem, as introduced in Section~\ref{sec:AED}, Eq.~\eqref{eq:inv}: We are seeking for inferring the  parameter $\mu$ in the right hand side  $f(x) = \mu \cos(\omega x)\sin(\omega y)$, for the unknown ground truth  $\mu_{gt} = 2\omega_{gt}^2$, $\omega_{gt} =\pi$ and the corresponding PDE solution simultaneously, with the $L^2$-loss ($k=l=0$) given in equation \eqref{eq:IP_ref}. The GF-PSM is applied for a Sobolev cubature with degree $n=100$ for the boundary and $n=30$ for the PDE loss. Benchmarks for the standard PINN and the ID-PINN are executed with the same number of training points.
\end{experiment}
\newcommand{\tabIPP}{
\begin{tabular}[t]{lllll}
 & \multicolumn{3}{c}{Approximation error} & Runtime (s)\\
& $\epsilon_{\mu}$ & $\epsilon_{1}$ & $\epsilon_{\infty}$ & \\
\hline
  PINN    & $4.63\cdot10^{-1}$ & $1.13\cdot 10^{-2}$ & $1.24 \cdot 10^{-1}$ & $t\approx 1592$\\
  ID-PINN & $2.14\cdot10^{-2}$ &$8.09\cdot10^{-4}$ & $1.52 \cdot 10^{-2}$& $t\approx 2184$ \\
  SC-PINN  & $3.0\cdot10^{-4}$ & $5.49\cdot10^{-4}$ & $1.01 \cdot 10^{-2}$ & $t \approx 103$ \\
  GF-PSM  & $\mathbf{5.8\cdot10^{-8}}$ & $\mathbf{6.0\cdot10^{-10}}$&$\mathbf{3.47\cdot10^{-9}}$ & $t \approx \mathbf{0.49}$ \\
& & &\\
& & &
\end{tabular}
}

\begin{table}[t]
\begin{minipage}[b]{.3\textwidth}%
\centering
\hspace*{-22pt}\includegraphics[width=1.0\textwidth]{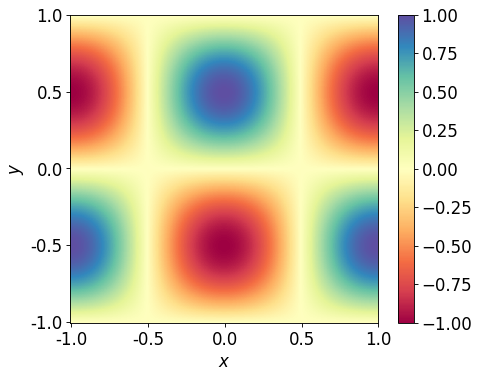}
\captionof{figure}{Solution for 2D inverse Poisson problem with $\omega_{gt} = \pi$.}%
\end{minipage}
\begin{minipage}[b]{.7\textwidth }%
\tabIPP
\caption{Errors for 2D Poisson inverse problem}\label{IP_QHO}%
\end{minipage}%
\end{table}

Table \ref{IP_QHO} reports the reached accuracy and the required runtimes. The GF-PSM outperforms all other methods by several orders of magnitude in accuracy for both the solution of the PDE, as well as the inferred parameter $\mu$. As discussed in Section~\ref{sec:AED} the analytic variation, Eq.~\eqref{eq:AS}, does not directly apply for this task and is, thus, omitted here. The exponentially fast convergence of the GF-PSM, Section~\ref{sec:AED}, is reflected in the required runtime being 4 orders of magnitude less than the PINN alternatives.
\subsection{Quantum Harmonic Oscillator in 2D}
We consider eigenvalue problem for the time-independent \emph{Quantum Harmonic Oscillator} in dimension $m=2$, which is a special case of the \emph{Schrödinger equation} with  linear potential $V(u(x)):=(x_1^2 + x_2^2)u(x)$, $u\in C^2(\Omega,\mathbb{R})$, see e.g., \cite{richard1980,griffiths2018}:
\begin{equation*}
   \li\{\begin{array}{rll}
       -\Delta u(x) + V(u(x))  &= \mu u(x)  &,  \forall x\in\Omega  \\
         u(x)  -g(x)     &= 0   &,  \forall x\in\partial\Omega\,,
\end{array}\re.
\end{equation*}
It is a classic fact, that the the eigenvalues are given by  $\mu = n_1 + n_2 + 1$, $n_1,n_2 \in \N$ with corresponding eigenfunctions
\begin{equation*}\label{eq:QHO}
g(x_1,x_2) = \frac{\pi^{-1/4}}{\sqrt{2^{n_1+n_2}n_1!n_2!}}e^{-\frac{(x_1^2+x_2^2)}{2}}H_{n_1}(x_1)H_{n_2}(x_2)\,,
\end{equation*}
whereas $H_n$ denotes the $n$-th \emph{Hermite polynomial}.

\newcommand{\tabE}{
\begin{tabular}[t]{llll}
  & \multicolumn{2}{c}{Approximation error} & Runtime (s)\\
$\mu =21$& $\epsilon_1$  & $\epsilon_\infty$  &  \\
\hline
  PINN & $6.97\cdot10^{-2}$ & $1.\cdot 10^{-3}$ & $t\approx 776$ \\
  ID-PINN & $4.29\cdot10^{-2}$ & $1.30\cdot 10^{-1}$ & $t\approx 948$ \\
  SC-PINN & $8.16\cdot10^{-4}$ & $7.27\cdot 10^{-3}$ & $t\approx 167$ \\
  GF-PSM & $ 1.6\cdot{10^{-8}}$ & $5.4\cdot10^{-8}$ & $t \approx 0.16$\\
   AD-PSM & $\mathbf{7.61\cdot10^{-13}}$ & $\mathbf{2.37\cdot10^{-12}}$ & $t \approx {\bf 0.07}$\\
   & & &\\
    $\mu =31$& $\epsilon_1$  & $\epsilon_\infty$  &  \\
    \hline
 GF-PSM& $1.09\cdot10^{-9}$ & $1.45\cdot10^{-8}$ & $t \approx 2.39$\\
 AD-PSM& $\mathbf{2.25\cdot10^{-9}}$ & $\mathbf{9.82\cdot10^{-9}}$ & $t \approx {\bf 1.07}$\\
    & & &\\
       & & &
\end{tabular}
}
\begin{table}[t!]
\begin{minipage}[b]{.4\textwidth}%
\centering
\hspace*{-80pt}\includegraphics[width=1.35\textwidth]{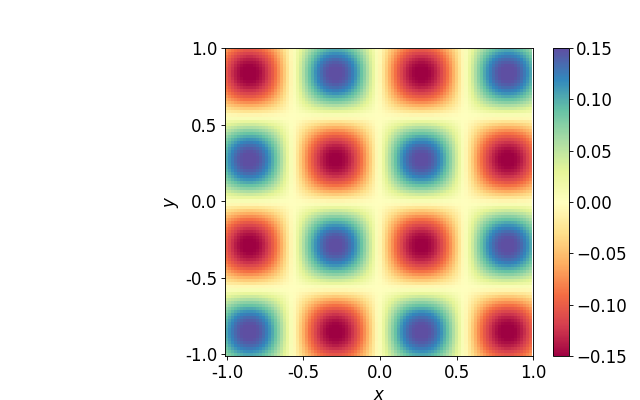}
\captionof{figure}{Solution of 2D QHO\\ forward problem with $\mu = 21$.}%
\end{minipage}
\begin{minipage}[b]{.60\textwidth }%
\tabE
\caption{Errors for 2D QHO forward problem with $\mu=21,31$.}\label{validationQ}%
\end{minipage}%
\end{table}
\begin{figure}[t!]%
      \includegraphics[width=1.0\textwidth]{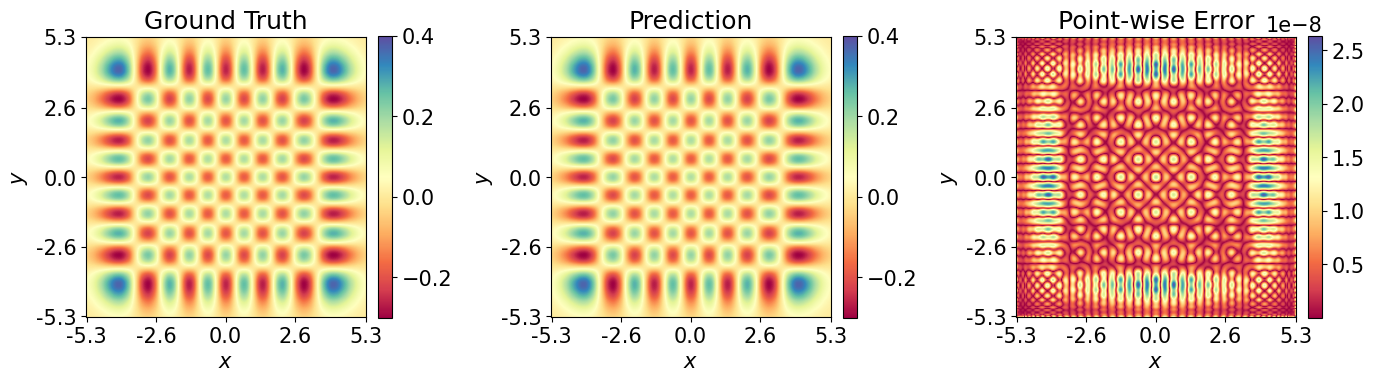}
      \caption{Solution for 2D QHO with $\mu = 31$ on $\Omega'=5.3\Omega$ due to AD-PSM.}\label{fig:QHO_AD}
\end{figure}

\begin{experiment}[QHO forward problem]
For solving the QHO forward problem with eigenvalue $\mu =21$ and extended domain $\Omega' = [-5.3, 5.3]$, GF-PSM and the AD-PSM  use Sobolev cubatures of degree $n=100$ for the boundary and $n=30$ for the PDE loss, whereas we choose $n=200$ and $n=50$ for eigenvalue $\mu =31$ on the standard hypercube $\Omega$, respectively.  The AD-PSM uses the by default chosen $H^{-1}(\Omega),*$ norm, while the GF-PSM was applied with weak $L^2$-loss, as in Eq.~\eqref{eq:wL}.
\end{experiment}
Results are reported in Table~\ref{validationQ}. SC-PINN was the only PINN method that gains reasonable results for $\mu =31$ and $\Omega = [-1,1]^2$. However, as in Section~\ref{sec:P} the PSMs-methods outperform SC-PINN in both runtime and accuracy performance. In the second scenario, $\mu =21$, $\Omega' = 5.3\Omega$, none of PINN  approaches was able to reach close approximations, while AD-PSM and GF-PSM do. AD-PSM performs best and its solution is visualised in Fig.~\ref{fig:QHO_AD}.

\begin{experiment}[QHO inverse problem]
Similar to Exp.~\ref{exp:PI}
we seek for inferring the  unknown eigenvalue $\mu$, set to $\mu_{gt}=9$, and the corresponding continuous approximation of the PDE solution simultaneously, with given data $\mathfrak{u}\in \mathbb{R}^{|A_{m,n}|}$ sampled on the Legendre grid by optimising the loss:
\begin{equation}
   \mathcal{L}[C,\mu] = \|\Delta Q_c + V(Q_c)-\mu Q_C\|^2_{L^2}+\|Q_C-\mathfrak{u}\|^2_{L^2}
\end{equation}
We choose a $n=50$ degree Sobolev cubature for the domain and $n=200$ on the boundary and compare it with the PINN and the ID-PINN for the same number of training points.
\end{experiment}

\newcommand{\tabQI}{
\begin{tabular}[t]{lllll}
 & \multicolumn{3}{c}{Approximation error}& Runtime (s)\\
& $\epsilon_{\mu}$ & $\epsilon_{1}$ & $\epsilon_{\infty}$\\
\hline
  PINN    & $6.01$ & $7.32\cdot 10^{-2}$ &$4.37\cdot 10^{-1}$ &  $t\approx 1414$\\
  ID-PINN & $6.21\cdot10^{-2}$ &$7.51\cdot10^{-3}$ & $9.40\cdot10^{-2}$&  $t\approx 1346$ \\
  SC-PINN  & $2.18\cdot10^{-4}$ & $5.68\cdot10^{-4}$ &  $1.39\cdot10^{-2}$& $t \approx 192$ \\
  GF-PSM  & $\mathbf{9.50\cdot10^{-11}}$ & $\mathbf{1.49\cdot10^{-12}}$ & $\mathbf{5.13\cdot10^{-10}}$ & $t \approx \mathbf{5}$ \\
& & &\\
& & & \\
& & &\\
& & &
\end{tabular}
}

\begin{table}[t!]
\vspace{-25pt}
\begin{minipage}[b]{.3\textwidth}%
\centering
\hspace*{-20pt}\includegraphics[width=1.0\textwidth]{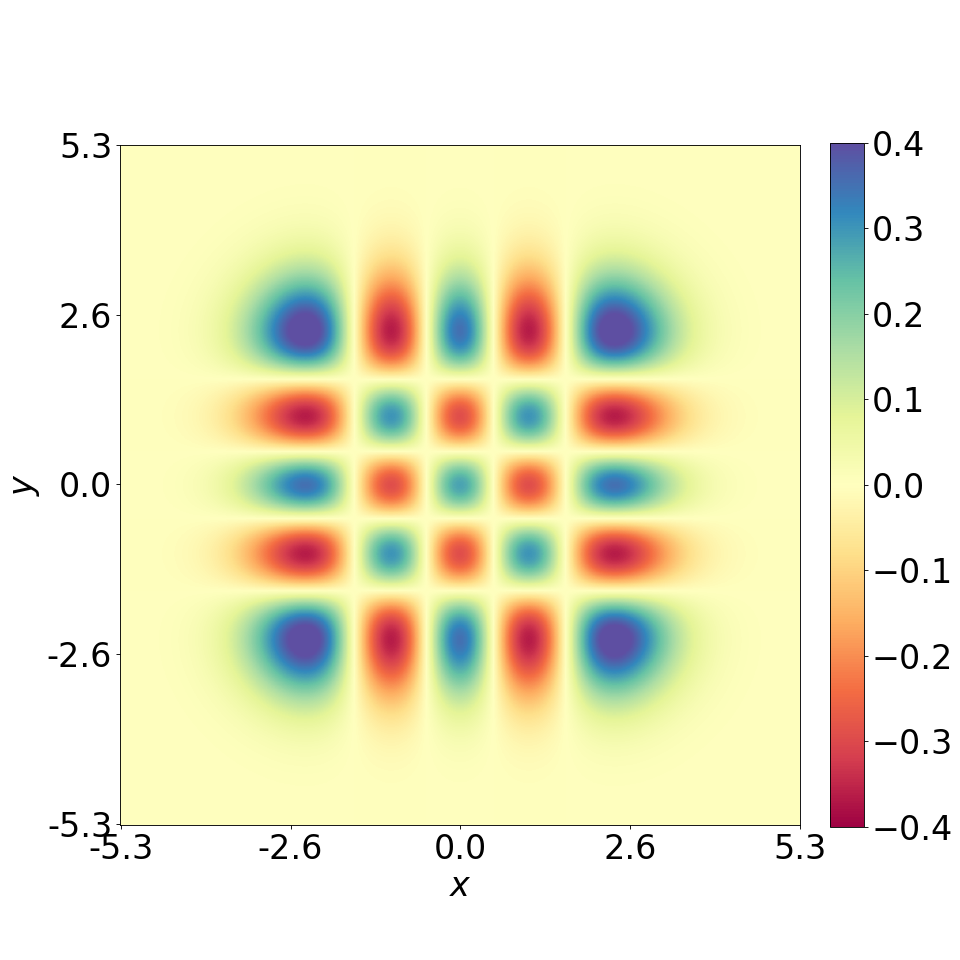}
\captionof{figure}{Solution for 2D QHO with $\mu_{gt} = 9$ on $\Omega' =5.3\Omega$.}%
\end{minipage}
\begin{minipage}[b]{.7\textwidth }%
\tabQI
\caption{Errors for 2D QHO inverse problem with $\mu_{gt}=9$}\label{IP_QHO_I}%
\end{minipage}%
\end{table}

As shown in Table \ref{IP_QHO_I} the GF-PSM outperforms the ID-PINN by several orders of magnitude in both accuracy and runtime. This reflects the strength and flexibility of the method when addressing linear inverse problems. While na\"ive, unconditioned Implicit-Euler implementations are inherently unstable the insights of Section~\ref{sec:AED} enable us to exploit the structure of the gradient flow to realize stable numerical integrators. Applying the PSM method to non-linear forward problems is our next demonstration task.

\subsection{2D Incompressible Navier Stokes equation}
We consider the incompressible 2D Navier Stokes equation as an example of a  non-linear PDE problem: Let $u = (u_1, u_2)$, $u\in C^2(\Omega,\R^2)$ be the vector velocity field and $p\in C^1(\Omega;\R)$ the scalar pressure field the equation becomes:
\begin{equation*}
   \li\{\begin{array}{rll}
       -\nu\Delta u(x,y) + (u(x,y)\cdot\nabla)u(x,y) + \nabla p(x,y)&= f(x,y)  &,  \forall (x,y)\in\Omega  \\
        \nabla\cdot u (x,y)&= 0 &,  \forall (x,y)\in\Omega \\
         u(x,y)  -g(x,y)     &= 0   &,  \forall (x,y)\in\partial\Omega\,,
\end{array}\re.
\end{equation*}
where
\begin{align*}
    f(x,y) &=2\nu\pi^2(u_1(x,y),u_2(x,y))+\pi \cos(\pi x)\cos(\pi y)(-u_1(x,y),u_2(x,y))\\
    & +\pi \sin (\pi x)\sin(\pi y)(u_2,-u_1) +\exp(\pi y)(1,\pi x)\,,\\
   g(x,y) &= [-\sin(\pi x)\cos(\pi y),\cos(\pi x)\sin(\pi y)]^T
\end{align*}
\begin{experiment}[Navier-Stokes Forward  and Inverse Problem]
We solve the Navier-Stokes forward problem by applying GF-PSM with $n=100$ and $n= 30$ degree Sobolev cubature for the boundary and the domain respectively. We set
the viscosity to $\nu = 0.05$ and use the analytic pressure field  $p=x\exp(\pi y)$ with Dirichlet boundary conditions.

The inverse problem seeks for inferring $\nu$ and the scalar pressure field $p$  for the ground truth viscosity $\nu_{\mathrm{gt}} = 0.05$ and $u_1 = -\sin(\pi x)\cos(\pi y)$, $u_2 = \cos(\pi x)\sin(\pi y)$. The errors $\epsilon_1$ and $\epsilon_\infty$ reported for this experiment, correspond to the predicted pressure against the ground truth one.
\end{experiment}

\begin{table}[h]
\begin{minipage}[b]{.3\textwidth}%
\centering
\hspace*{-20pt}\includegraphics[width=1.2\textwidth]{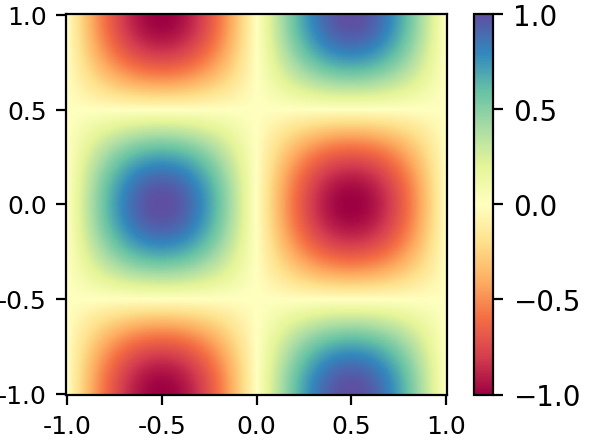}
\captionof{figure}{Solution $u_1$.}%
\end{minipage}
\begin{minipage}[b]{.70\textwidth }%
\begin{tabular}[t]{lllll}
 & &\multicolumn{2}{c}{Approximation error}& Runtime (s)\\
Forward Problem & & $\epsilon_{1}$ & $\epsilon_{\infty}$\\
\hline
  GF-PSM &$u_1$  & $3.31 \cdot10^{-10}$ & $2.35\cdot10^{-9}$ & $t \approx 405.22$ \\
  GF-PSM &$u_2$  & $3.28 \cdot10^{-10}$ & $2.35\cdot10^{-9}$ &  $t \approx 405.22$\\
& & &\\
& & &\\
& & &
\end{tabular}
\caption{Approximation errors of the forward problem. }\label{NVE}%
\end{minipage}
\end{table}
\begin{table}[h]
\center
\begin{tabular}[t]{lllll}
 & \multicolumn{3}{c}{Approximation error}& Runtime (s)\\
Inverse Problem & $\epsilon_{\nu}$ & $\epsilon_{1}$ & $\epsilon_{\infty}$\\
\hline
  GF-PSM  & $2.91\cdot10^{-16}$ & $2.63\cdot10^{-14}$ & $1.21\cdot10^{-11}$& $t \approx 0.79$
\end{tabular}
\caption{Approximation errors of the inverse problem.}
\label{tab:inv}
\end{table}

While none of the PINN approaches was able to address the problem reasonably the PSM methods reach similar accuracy as in the prior (linear) experiments, as reported in  Tables \ref{NVE},\ref{tab:inv}.


We summarise the experimental and theoretical findings in the concluding thoughts below.
\section{Conclusion}

We introduced a novel variational spectral method solving linear, non-linear, forward and inverse PDE problems. In contrast to neural network - PINN approaches  Chebyshev polynomials surve as a polynomial surrogate model - PSM, maintainig the same flexibility as PINNs.

Based on our prior work \cite{cardona2022replacing}, we  gave  weak PDE formulations, resting on the novel
Sobolev cubatures approximating general Sobolev norms. Allowing us to formulate and compute  the resulting finite-dimensional gradient flow for finding the optimal coefficients for the PSMs, in the case of linear PDEs, we could even derive the analytical solution of the gradient flow.
In particular, the  resulting efficient computation of the negative order dual Sobolev norm $\|\cdot\|_{\hmk,*}$ was demonstrated to perform best compared to the alternative formulations.
 While we  meanwhile deepened the theoretical insights, presented here, to deliver the optimal choice of the Sobolev norm beforehand these subjects are part of a follow-up study. This includes a relaxation of the Sobolev cubatures, resisting the curse of dimensionality when addressing higher dimensional problems.

In summary, the PSMs methods outperformed all other benchmark methods by far,  showing the superiority  in runtime and accuracy performance  of the PSMs formulation on the whole spectrum of the considered problems.
Since the PSMs offer the same flexibility and capabilities of PINNs, we propose to extend the presented approach in order to learn PDE solutions for ranges of boundary conditions, parameters (like diffusion constants) or dynamic time ranges. Because the gain in efficiency allowed to compute the presented benchmarks \emph{without High Performance Computing} (HPC) on a local machine, we expect so far non-reachable high-dimensional $\dim \geq 3 $, strongly varying PDE problems, appearing for instance for dynamic phase space simulations,  to become solvable when being addressed by a parallelised HPC version of the current implementation \cite{PDE_repo}.

\bibliography{REF.bib}
\bibliographystyle{icml2022}

\section*{Appendix}
The result provided in Theorem~\ref{thm_conv} is a known fact and could be also found for example in \cite{Karimi} in a more general setting. We prove it by combining the following lemmas. Given a differentiable $\lambda$-convex truncated loss
$\Lc : \R^{|A_{m,n}|}\lo \R^+$, $m,n \in \N$, as in Theorem~\ref{thm_conv}, inducing the gradient descent ODE
\begin{equation*}
 \partial_t C(t) = - \nabla \Lc(Q_{C(t)})      \quad  \,, C(0)= C_0\,,
\end{equation*}
where $C_0\in \mathbb{R}^{|A_{m,n}|}$ is some initial guess of the coefficients.
The Implicit Euler discretisation of the ODE  is given by
\begin{equation}\label{eq:EUL}
    C_{n+1} = C_n -\tau \nabla{L}[C_{n+1}]\,,
\end{equation}
where  $\tau\in \mathbb{R}$ is the learning rate.
We will use the following two definitions:
\begin{definition}
A functional $\mathcal{F}:\mathbb{R}^{|A_{m,n|}}\rightarrow \mathbb{R}$ is convex if:
\begin{equation}
    \mathcal{F}[tx+(1-t)y]\leq t\mathcal{F}[x]+(1-t)\mathcal{F}[y],
\end{equation}
it is called strictly convex, if the inequality is strict.
\end{definition}
\begin{definition}
A functional $\mathcal{F}:\mathbb{R}^{|A_{m,n}|}\rightarrow \mathbb{R}$ is coercive if:
\begin{equation}
    \lim\limits_{||u||\rightarrow \infty}\mathcal{F}[u]= \infty
\end{equation}
\end{definition}
\begin{lemma}\label{lemma_1}
Let the assumptions of Theorem~\ref{thm_conv} be fulfilled then the following estimate applies:
\begin{equation*}
    \frac{\lambda}{2}\|C_n-C_\infty\|^2\leq \Lc[C_n]-\Lc[C_\infty]\leq \frac{1}{2\lambda}\|\nabla\Lc[C_n]\|^2\,.
\end{equation*}
\end{lemma}
\begin{proof}
We prove the first inequality by rephrasing the $\lambda$- convexity property,Eq.~\eqref{eq:strong_conv}. Let $\gamma_t := t x+(1-t)y$, then $\Lc = \Lc(x)$ is $\lambda$-convex if
\begin{equation*}
    \Lc[\gamma_t]\leq t \Lc[x] + (1-t) \Lc[y]-\frac{\lambda}{2}t(1-t)\|x-y\|^2\,.
\end{equation*}
By replacing $x$ and $y$ with $C_n$ and $C_\infty$, respectively, and re-arranging, we obtain:
\begin{equation*}
\frac{\lambda}{2}t(1-t)\|C_n-C_\infty\|^2\leq t(\Lc[C_n]-\Lc[C_\infty]) +\Lc[C_\infty]-\Lc[\gamma_t]
    \leq t(\Lc[C_n]-\Lc[C_\infty])\,,
\end{equation*}
where we used the minimality of $C_\infty$ for the last inequality. Dividing by $t$ and taking the limit for $t\rightarrow 0$ yields the first inequality of Lemma~\ref{lemma_1}. The second inequality follows directly from the
$\lambda$-convexity, Eq.~\eqref{eq:strong_conv}, implying
\begin{equation*}
    \Lc[C_n]-\Lc[C_\infty]\leq -\nabla\Lc[C_n]^T(C_\infty-C_n)-\frac{\lambda}{2}\|C_\infty-C_n\|^2,
\end{equation*}
We set $\mathcal{F}[C^\infty]:=\nabla\Lc[C_n]^T(C_\infty-C_n)+\frac{\lambda}{2}\|C_\infty-C_n\|_2^2$ and realise that $\mathcal{F}$ is a coercive, strictly convex functional with respect to $C^\infty$.
Hence, the uniquely determined minimum $C^*_\infty$ is given by:
\begin{equation*}
    \nabla\mathcal{F} \stackrel{!}{=} 0 \iff (C^*_\infty-C_n) = -\frac{1}{\lambda}\nabla\Lc[C_n].
\end{equation*}
In light of this fact, we can bound $-\nabla\mathcal{F}$ by
\begin{equation*}
     \Lc[C_n]-\Lc[C_\infty]\leq (\frac{1}{\lambda}-\frac{1}{2\lambda})\|\nabla\Lc[C_n]\|^2\,,
\end{equation*}
yielding the desired result.
\end{proof}
The following lemma provides the monotonicity  property of the gradient flow, being a necessary ingredient for proving the exponential convergence.
\begin{lemma}\label{lemma_nd}
Let the assumptions of Theorem~\ref{thm_conv} be fulfilled the the following estimate holds:
\begin{equation*}
    \Lc[C_{n-1}]-\Lc[C_{\infty}]\geq (1+\lambda \tau)^2(\Lc[C_{n}]-\Lc[C_{\infty}])
\end{equation*}
\end{lemma}
\begin{proof}
Due to the $\lambda$-convexity and the Implicit Euler update, Eq.~\eqref{eq:EUL}, we realise that:
\begin{align*}
    \Lc[C_{n-1}]&\geq \Lc[C_n] +\nabla \Lc[C_n](C_{n-1}-C_n)+\frac{\lambda}{2}\|C_{n-1}-C_n\|^2\\
    &= \Lc[C_n] +\tau(\frac{\tau\lambda}{2}+1)\|\nabla \Lc[C_n]\|^2\,.
\end{align*}
Due to Lemma~\ref{lemma_1} we further conclude
\begin{equation}
    \Lc[C_{n-1}]\geq \Lc[C^{n}] + 2\lambda\tau(\frac{\tau\lambda}{2}+1)(\Lc[C_n]-\Lc[C_\infty])\,.
\end{equation}
Adding $-\Lc[C_\infty]$ at both sides provides the claim.
\end{proof}
\begin{lemma}\label{lemma_f} Let the assumptions of Theorem~\ref{thm_conv} be fulfilled and define $\hat{\lambda}:= \frac{1}{\tau}\log(1+\lambda\tau)$.
Then the sequence:
\begin{equation*}
    \Delta^n\Lc:=\Lc[C_n]-\Lc[C_\infty],
\end{equation*}
 decreases monotonically with an exponential rate of $e^{-2\hat{\lambda}\tau n}$, i.e.
 \begin{equation}
     \Delta^n\Lc\leq e^{-2\hat{\lambda}\tau n}(\Lc[C_0]-\Lc[C^\infty])
 \end{equation}
\end{lemma}
\begin{proof}
Due to Lemma \eqref{lemma_nd} we compute
\begin{align*}
     e^{2\hat{\lambda}\tau n}(\Lc[C_n]-\Lc[C_\infty]) &= (1+\lambda\tau)^{2n}(\Lc[C_n]-\Lc[C_\infty])
     \\&\leq (1+\lambda\tau)^{2(n-1)}(\Lc[C_{n-1}]-\Lc[C_\infty])\\
     & \cdots\\
     & \leq\Lc[C_0]-\Lc[C_\infty]\,.
\end{align*}
\end{proof}
\begin{proof}[Proof of Theorem~\ref{thm_conv}]
Theorem \eqref{thm_conv} now follows by combing
Lemma \eqref{lemma_1} and \eqref{lemma_f} yielding:
\begin{equation}
    \frac{1}{\lambda}\|C_n-C_\infty\|_2^2\leq\Lc[C_n]-\Lc[C_\infty]\leq e^{-2\hat{\lambda}\tau n}(\Lc[C_0]-\Lc[C_\infty])\,.
\end{equation}
Thus, for $\tau\rightarrow 0$, it follows by the definition of $\hat{\lambda}$ that $\hat{\lambda}\rightarrow \lambda$ and $C_n\rightarrow C(t)$, with $t = n\tau$ due to the continuity of $C=C(t)$ inherited from the differentiability of $F$. Hence, the continuity of the norm implies the statement.
\end{proof}

\begin{remark}
Lemma \ref{lemma_nd} implies that also the Implicit Euler discretised gradient flow, converges exponentially fast.
\end{remark}

\end{document}